\documentclass[12pt]{book}
\usepackage{amssymb,amsmath,color,graphicx,stmaryrd}
\title{Two Minicourses on Analytic Microlocal Analysis}
\author{Michael Hitrik \\\small Department of Mathematics \\\small University of California
\\\small Los Angeles \\\small CA 90095-1555, USA\\\small hitrik@math.ucla.edu \and
Johannes Sj\"ostrand\\\small IMB, Universit\'e de Bourgogne \\\small 9 avenue Alain Savary BP 47870 \\\small 21078 Dijon Cedex, France
\\\small and UMR 5584 CNRS \\\small johannes.sjostrand@u-bourgogne.fr}
\date{}

\def\wrtext#1{\relax\ifmmode{\leavevmode\hbox{#1}}\else{#1}\fi}
\def\abs#1{\left|#1\right|}
\def\begeq{\begin{equation}}
\def\endeq{\end{equation}}

\def\iint{\int\hskip -2mm\int}

\def\Re{{\rm Re\,}}
\def\Im{{\rm Im\,}}

\def\part#1{\frac{\partial}{\partial #1}}

\def\norm#1{||\,#1\,||}

\newcommand{\real}{\mbox{\bf R}}
\newcommand{\comp}{\mbox{\bf C}}

\renewcommand{\Re}{\mbox{\rm Re\,}}
\renewcommand{\Im}{\mbox{\rm Im\,}}

\renewcommand{\exp}{\mbox{\rm exp\,}}

\newtheorem{dref}{Definition}[section]
\newtheorem{lemma}[dref]{Lemma}
\newtheorem{theo}[dref]{Theorem}
\newtheorem{prop}[dref]{Proposition}

\newtheorem{example}[dref]{Example}

\newenvironment{proof}{\vspace{.3cm}\noindent{{\em Proof:}}}{\hfill$\Box$}

\begin{document}

\maketitle
\pagestyle{plain}

\vspace*{1cm}
\noindent
{\bf Abstract}: These notes correspond roughly to the two minicourses prepared by the authors for the workshop on Analytic Microlocal Analysis, held at Northwestern
University in May 2013. The first part of the text gives an elementary introduction to some global aspects of the theory of metaplectic FBI transforms, while the
second part develops the general techniques of the analytic microlocal analysis in exponentially weighted spaces of holomorphic functions.

\bigskip
\noindent

{\begin{center}
\large{In memory of Lars G\aa{}rding and Lars H\"ormander}
\end{center}}

\tableofcontents

\chapter{Introduction to metaplectic FBI transforms}

\section{Introduction}
The metaplectic Fourier-Bros-Iagolnitzer (FBI) transform allows one to pass from the standard Hilbert space $L^2(\real^n)$ to an exponentially weighted space
of holomorphic functions on $\comp^n$. Such transforms occur under various other names in the literature, such as the Bargmann, Segal, Gabor, and
wave packet transforms, and from the general point of view of microlocal analysis, these can all be viewed as Fourier integral operators with complex phase.
In this part of the text, the connection to analytic microlocal analysis will be emphasized, and we shall therefore refer to these transforms as FBI transforms,
as they were used by J. Bros and D. Iagolnitzer to give a definition of the analytic wave front set. 
Pseudodifferential operators can be transported to the FBI transform side, and in this way, one obtains some flexible and
powerful techniques for their analysis, particularly in the analytic case. In this chapter, we give an elementary introduction to the theory of
metaplectic FBI transforms. In Section 1.2 we discuss aspects of the geometry of positive  complex Lagrangian planes and some closely related complex canonical
transformations, following Appendix A of~\cite{CaGrHiSj} and Chapter 11 of~\cite{SjAst}. In Section 1.3, following~\cite{Sj95},~\cite{Sj02}, we introduce
metaplectic FBI transforms, derive a representation for the Bergman projection and establish the unitarity of the FBI transform between $L^2(\real^n)$
and a suitable weighted space of holomorphic functions on $\comp^n$. See also~\cite{H91},~\cite{Zworski}. Section 1.4 is concerned with pseudodifferential
operators on the FBI transform side. We discuss their mapping properties and prove the metaplectic Egorov theorem, finishing with a brief discussion of the
case of pseudodifferential operators with holomorphic symbols. Our presentation here follows~\cite{Sj95} and~\cite{Sj02} closely.

\medskip
\noindent
{\bf Acknowledgement}. The first author would like to thank Michael Hall for providing him with some notes which were used in the preparation of the present text.

\setcounter{equation}{0}
\section{Complex symplectic linear algebra. Positivity}
\setcounter{equation}{0}
We shall work in the complex space $\comp^{2n} = \comp^n_x \times \comp^n_{\xi}$, which is equipped with the complex symplectic (2,0)-form
\begeq
\label{eq1.1}
\sigma = \sum_{j=1}^n d\xi_j \wedge dx_j,\quad (x,\xi)\in \comp^{2n}.
\endeq
The form $\sigma$ is non-degenerate and closed, and we can write
\begeq
\label{eq1.2}
\sigma(X,Y) = JX\cdot Y, \quad J = \begin{pmatrix} 0 & 1 \\ -1 & 0 \end{pmatrix}, \quad X,Y\in \comp^{2n}.
\endeq
Here and in what follows we shall use the complex bilinear scalar product on $\comp^k$, given by $X\cdot Y = \sum_{j=1}^k X_j Y_j$.

\medskip
\noindent
The corresponding real 2-forms
\begeq
\label{eq1.3}
\operatorname{Re\,}\sigma = \frac{\sigma + \bar \sigma}{2}, \quad \operatorname{Im\,}\sigma = \frac{\sigma - \bar \sigma}{2i}.
\endeq
are closed and non-degenerate, and hence give rise to real symplectic structures on $\comp^{2n}$.

\begin{dref}
A complex linear map $\kappa: \comp^{2n} \to \comp^{2n}$ is called a complex canonical transformation if
\begeq
\label{eq1.4}
\sigma(\kappa(X),\kappa(Y)) = \sigma(X,Y),\quad X,Y\in \comp^{2n}.
\endeq
\end{dref}

\medskip
\noindent
If $\kappa: \comp^{2n} \rightarrow \comp^{2n}$ is a complex canonical transformation, then $\kappa$ preserves the complex volume form $\sigma^n/n!$
on $\comp^{2n}$, and therefore ${\rm det}\, \kappa = 1$. If $n=1$, the converse is also true.

\medskip
\noindent
Let us consider the following configuration: Let $\Sigma \subseteq \comp^{2n}$ be a real subspace which is {\it I-Lagrangian} in the sense that
$\operatorname{dim}_\mathbf{R} \Sigma = 2n$ and $\operatorname{Im\,}\sigma|_\Sigma = 0$. Assume also that $\Sigma$ is {\it R-symplectic}:
$\operatorname{Re\,}\sigma|_\Sigma$ is non-degenerate. Such a subspace is automatically maximally totally real, $\Sigma \cap i\Sigma = \{0\}$, and we can write
$$
\comp^{2n} = \Sigma \oplus i\Sigma.
$$
Let $\Gamma = \Gamma_\Sigma: \comp^{2n} \to \comp^{2n}$ be the unique antilinear map such that $\Gamma|_\Sigma = 1$. Clearly, we have
\begeq
\label{eq1.41}
\sigma(\Gamma X, \Gamma Y) = \overline{\sigma(X,Y)},\quad X,Y\in \comp^{2n}.
\endeq

\bigskip
\noindent
{\bf Examples}.

\begin{enumerate}
\item $\Sigma = \real^{2n}$, $\Gamma X = \bar{X}$, the complex conjugation.
\item Let $\Phi$ be a real valued quadratic form on $\comp^n_x$, such that the Levi matrix,
$\partial_{\bar x} \partial_x \Phi = (\partial_{\bar x_j}\partial_{x_k}\Phi)_{j,k=1}^n$, is non-degenerate.
\end{enumerate}

Let us set
\begeq
\label{eq1.5}
\Sigma = \Lambda_{\Phi} :=  \left\{ \left(x,\frac{2}{i} \frac{\partial \Phi}{\partial x}(x) \right);\,\, x\in \comp^n \right\}.
\endeq
We claim that the linear subspace $\Sigma$ is I-Lagrangian and R-symplectic. Indeed, using $x\in \comp^n$ to parametrize $\Lambda_{\Phi}$,
we get
\begeq
\label{eq1.6}
\sigma |_{\Lambda_{\Phi}} = \sum_{k=1}^n d\left( \frac{2}{i} \frac{\partial \Phi}{\partial x_k} \right)\wedge dx_k =
\sum_{j,k=1}^n \frac{2}{i}\frac{\partial^2 \Phi}{\partial \bar x_j \partial x_k}d\bar x_j \wedge dx_k.
\endeq
Using only the fact that $\Phi$ is real, we see that $\sigma |_{\Lambda_{\Phi}}$ is real, so that $\Lambda_{\Phi}$ is I-Lagrangian. Since the Levi form of $\Phi$
is non-degenerate, (\ref{eq1.6}) also shows that $\sigma |_{\Lambda_{\Phi}}$ is non-degenerate.

\medskip
\noindent
Let us now describe the involution $\Gamma|_{\Lambda_\Phi}$ explicitly. We have
\begeq
\label{eq1.7}
\Phi(x) = \frac{1}{2} \Phi''_{xx} x\cdot x + \Phi''_{\bar{x} x}x\cdot \bar{x} +
\frac{1}{2} \Phi''_{\bar{x} \bar{x}}\bar{x}\cdot \bar{x},
\endeq
and therefore,
\begeq
\label{eq1.8}
\Lambda_{\Phi} =  \left\{ \left(x,\frac{2}{i} \left(\Phi''_{x x}x + \Phi''_{x \bar x}\bar x \right)\right);\,\, x\in \comp^n \right\}.
\endeq
Using that $\Gamma_{\Lambda_{\Phi}}(X+iY) = X-iY$, $X,Y\in \Lambda_{\Phi}$, we see that $\Gamma = \Gamma_{\Lambda_{\Phi}}$ is given by
\begeq
\label{eq1.9}
\left(y, \frac{2}{i} \left(\Phi''_{xx} y + \Phi''_{x \bar x}\bar x\right) \right) \mapsto
\left(x, \frac{2}{i}\left(\Phi''_{xx} x + \Phi''_{x \bar x}\bar y\right) \right)
\endeq
Notice that the map (\ref{eq1.9}) is well-defined since ${\rm det}\, \left(\Phi''_{\bar x x}\right) \neq 0$.

\bigskip
\noindent
Now let $\Lambda\subseteq \comp^{2n}$ be a $\comp$-Lagrangian subspace, i.e. a complex linear subspace such that $\operatorname{dim}_\mathbf{C} \Lambda = n$ and
$\sigma|_\Lambda = 0$. If $\Sigma \subseteq \comp^{2n}$ is I-Lagrangian, R-symplectic as above, with the associated involution $\Gamma$, we can introduce the
Hermitian form
\begeq
\label{eq1.10}
b(X,Y) = \frac{1}{i}\sigma (X,\Gamma Y),\quad (X,Y) \in \Lambda \times \Lambda.
\endeq
Here the Hermitian property, $\overline{b(X,Y)} = b(Y,X)$, follows from (\ref{eq1.41}).

\medskip
\noindent
{\it Remark}. When $\Sigma = \real^{2n}$, the Hermitian form (\ref{eq1.10}) was introduced in~\cite{H71}. The general case was considered in~\cite{SjAst}.

\medskip
\noindent
\begin{prop}
\label{prop1.1}
The form $b$ is non-degenerate if and only if the subspaces $\Lambda$ and $\Sigma$ are transversal, i.e. $\Lambda \cap \Sigma = \{0\}$.
\end{prop}
\begin{proof}
Consider the radical of $b$,
$$
{\rm Rad}\,(b) = \{X\in \Lambda; b(X,Y) = 0\,\,\wrtext{for all}\,\,\, Y\in \Lambda\}.
$$
If $0\neq X\in {\rm Rad}\,(b)$, then $\sigma(\Gamma X,Y) = 0$ for all $Y\in \Lambda$, and therefore, $\Gamma X \in \Lambda$, since $\Lambda$ is Lagrangian. We see,
using the fact that $\Gamma$ is an antilinear involution, that the vectors $(1/2)\left(X + \Gamma X\right)$ and $(1/2i)\left(X - \Gamma X\right)$ both belong to
$\Lambda \cap \Sigma$, and at least one of them is $\neq 0$, so that $\Lambda \cap \Sigma \neq \{0\}$. Conversely, $\Lambda \cap \Sigma \subseteq {\rm Rad}\, (b)$,
and the result follows.
\end{proof}

\bigskip
\noindent
\begin{example} \end{example} Let $\Sigma = \real^{2n}$ and assume that $\Lambda$ is transversal to the fiber $F = \left\{(0,\xi);\,\,\, \xi\in \comp^{n} \right\}$,
$\Lambda \cap F = \{0\}$. Then necessarily, $\Lambda= \Lambda_\varphi$ is of the form $\xi = \varphi'(x) = \varphi'' x$, where $\varphi$ is a holomorphic
quadratic form on $\comp^n_x$. We can compute the form $b$ explicitly using this representation of $\Lambda$. When $X = (x,\varphi'' x)\in \Lambda$, we get,
using (\ref{eq1.10}),
\begeq
\label{eq1.11}
\frac{1}{2}b(X,X) = \left(\operatorname{Im\,}\varphi''\right) x\cdot \bar x.
\endeq
Here
$$
{\rm Im}\, \varphi'' = \frac{1}{2i}\left(\varphi'' - \left(\varphi''\right)^*\right).
$$

\medskip
\noindent
\begin{dref}
Let $\Lambda\subseteq \comp^{2n}$ be $\comp$-Lagrangian and let $\Sigma \subseteq \comp^{2n}$ be I-Lagrangian, R-symplectic, with the involution $\Gamma$.
We say that $\Lambda$ is \emph{$\Sigma$-positive (negative)} if the Hermitian form $b$ is positive definite (negative definite) on $\Lambda$.
\end{dref}

\medskip
\noindent
\begin{prop}
Let $\Sigma = \real^{2n}$. Then $\Lambda$ is $\Sigma$--positive if and only if $\Lambda = \Lambda_\varphi$, where $\operatorname{Im\,}\varphi'' > 0$.
\end{prop}
\begin{proof}
If $\Lambda = \Lambda_{\varphi}$ with $\operatorname{Im\,}\varphi'' > 0$, then in view of (\ref{eq1.11}), we see that $\Lambda$ is $\Sigma$--positive. Conversely, if
$\Lambda$ is $\Sigma$-positive, then $\Lambda$ is transversal to the fiber $F$, so that $\Lambda = \Lambda_{\varphi}$, and Example 1.2.3 applies again.
\end{proof}

\begin{prop}
The set $\{\Lambda\subseteq \comp^{2n};\,\, \Lambda\text{ is }\comp-\text{Lagrangian and }\Lambda\text{ is }\Sigma-\text{positive}\}$ is a connected component
in the set of all $\comp$-Lagrangian spaces that are transversal to $\Sigma$.
\end{prop}
\begin{proof}
After applying a suitable linear complex canonical transformation, we may assume that $\Sigma = \real^{2n}$. Proposition 1.2.5 shows then that the set of all
$\Sigma$-positive $\comp$-Lagrangian spaces is a connected (even convex) and open subset of the set of all $\comp$-Lagrangian spaces that are transversal to
$\Sigma$. It is also closed, for if $\Lambda$ is a $\comp$-Lagrangian space transversal to $\Sigma$, such that the form $b$ is positive semi-definite on
$\Lambda$, then $b$ is necessarily positive definite on $\Lambda$, in view of Proposition 1.2.2. We conclude that the set of all $\Sigma$-positive
$\comp$-Lagrangian spaces is a component in the set of all $\comp$-Lagrangian spaces that are transversal to $\Sigma$.
\end{proof}

\bigskip
\noindent
Let us return to the situation where $\Sigma = \Lambda_{\Phi}$, with $\Phi$ being a real quadratic form on $\comp^n_x$. Assume that the Levi form of
$\Phi$ is positive definite,
\begeq
\label{eq1.12}
\sum_{j,k=1}^n \frac{\partial^2 \Phi}{\partial \bar x_j \partial x_k} \bar \xi_j \xi_k > 0, \quad \forall 0 \neq \xi \in \comp^n,
\endeq
i.e. the quadratic form $\Phi$ is strictly pluri-subharmonic.

\begin{prop}
The fiber $F = \{(0,\eta); \,\,\, \eta\in \comp^n\}$ is $\Lambda_\Phi$-negative.
\end{prop}

\begin{proof}
Using (\ref{eq1.9}) we see that $\Gamma(0,\eta) = (x,\xi)$, where $\xi = \frac{2}{i} \Phi''_{xx} x$, $\eta = \frac{2}{i}\Phi''_{x\bar x}\bar x$, which implies that
$$
\frac{1}{i} \sigma((0,\eta),(x,\xi)) = \frac{1}{i}\eta\cdot x = -2 \Phi''_{x\bar x}\bar x\cdot x \leq \frac{-1}{C}\abs{x}^2 \leq
-\frac{1}{\widetilde{C}}\abs{\eta}^2.
$$
\end{proof}

\medskip
\noindent
Now the space $\Gamma(F): \xi  = \frac{2}{i}\Phi''_{x x}x = \frac{1}{i}\partial_x\left( \Phi''_{xx}x\cdot x \right)$ is $\comp$-Lagrangian and
$\Lambda_\Phi$-positive. Let us write
$$
\Phi(x) = \Phi_{{\rm plh}}(x) + \Phi_{{\rm herm}}(x),
$$
where
$$
\Phi_{\rm plh}(x) = \operatorname{Re\,}\left(\Phi''_{xx}x\cdot x\right)
$$
is the pluri-harmonic part, and
$$
\Phi_{\rm herm}(x) = \Phi''_{\bar x x}x\cdot \bar x
$$
is the positive definite Hermitian part. Using that
$$
\partial_x \left(\Phi''_{xx} x\cdot x\right) = 2\partial_x \Phi_{{\rm plh}}(x),
$$
we conclude that $\Gamma(F)$ is of the form $\Lambda_{\Phi_{{\rm plh}}}$, where $\Phi(x) - \Phi_{\rm plh}(x)\sim \abs{x}^2$.

\begin{prop}
Assume that $\partial_{\bar x} \partial_x \Phi > 0$. A $\comp$-Lagrangian space $\Lambda$ is $\Lambda_\Phi$-positive if and only if
$\Lambda = \Lambda_{\widetilde \Phi}$, where $\widetilde \Phi$ is pluri-harmonic quadratic and $\Phi - \widetilde \Phi \sim \abs{x}^2$.
\end{prop}

\begin{proof}
If $\widetilde{\Phi}$ is pluri-harmonic quadratic and $\Phi - \widetilde{\Phi} > 0$ then clearly, $\Lambda_{\widetilde{\Phi}}$ is $\comp$-Lagrangian and
transversal to $\Lambda_{\Phi}$. It follows that the set
$$
\{\Lambda_{\widetilde \Phi}; \,\, \widetilde\Phi \text{ pluri-harmonic },\,\, \Phi - \widetilde \Phi > 0\}
$$
is an open connected subset of the set of all $\comp$--Lagrangian spaces that are transversal to $\Lambda_{\Phi}$. It is also closed, for if $\widetilde{\Phi}$ is
pluri-harmonic, $\Phi - \widetilde{\Phi}\geq 0$, and $\Lambda_{\widetilde{\Phi}}$ is transversal to $\Lambda_{\Phi}$, then the quadratic form
$\Phi - \widetilde{\Phi}$ is necessarily positive definite. (The transversality forces a non-strict inequality to become strict.) It follows that the set
$\{\Lambda_{\widetilde \Phi}; \,\, \widetilde\Phi \text{ pluri-harmonic },\,\, \Phi - \widetilde \Phi > 0\}$ is a connected component of the set of all
$\comp$-Lagrangian spaces that are transversal to $\Lambda_\Phi$. It contains
$\Lambda_{\Phi_{\rm plh}}$, as we saw above, which is $\Lambda_\Phi$-positive. An application of Proposition 1.2.6 allows us to conclude the proof.
\end{proof}

\bigskip
\noindent
{\bf Example}. Let $\Sigma = \real^{2n}$, and let $\Lambda_\pm \subseteq \comp^{2n}$ be $\comp$-Lagrangian spaces such that $\Lambda_+$ is positive and $\Lambda_-$
is negative, with respect to $\Sigma$. Let us verify that there exists a holomorphic quadratic form $\varphi(x,y)$ on $\comp_x^n\times \comp^n_y$ such that
\begeq
\label{eq1.12.1}
{\rm det}\,\varphi''_{xy} \neq 0,\quad \operatorname{Im\,}\varphi''_{yy} > 0,
\endeq
and such that the complex linear canonical transformation
$$
\kappa_{\varphi}: \comp^{2n}\ni (y,-\varphi'_y(x,y)) \mapsto (x,\varphi'_x(x,y))\in \comp^{2n}
$$
satisfies
\begeq
\label{eq1.12.2}
\kappa_{\varphi}(\Lambda_+) = \{(x,0);\, x\in \comp^n\},
\endeq
and
\begeq
\label{eq1.12.3}
\kappa_{\varphi}(\Lambda_-) = \{(0,\xi);\, \xi\in \comp^n\}.
\endeq

\medskip
\noindent
When showing the existence of the quadratic form $\varphi(x,y)$, let us recall from Proposition 1.2.5 that $\Lambda_{\pm}$ has the form $\eta = F_{\pm}y$, where
$F_{\pm}$ is a complex symmetric matrix such that $\pm {\rm Im}\,F_{\pm} > 0$. Looking for $\varphi$ in the form
$$
\varphi(x,y) = \frac{1}{2} Ax\cdot x + Bx\cdot y + \frac{1}{2}Cy\cdot y,
$$
where the matrices $A$ and $C$ are symmetric and $B$ is bijective, we observe first that (\ref{eq1.12.3}) is equivalent to the fact that
$$
\kappa_{\varphi}^{-1}(\{(0,\xi);\, \xi\in \comp^n\}) = \{(y,-Cy);\, y\in \comp^n\} = \Lambda_-,
$$
so we must have
\begeq
\label{eq1.12.31}
C = -F_-.
\endeq
The second condition in (\ref{eq1.12.1}) is then satisfied, and we also see that
\begin{multline}
\label{eq1.12.4}
\kappa_{\varphi}^{-1}(\{(x,0);\,x\in \comp^n\}) = \{(y,-Bx -Cy);\, Ax + B^t y =0\} \\
= \{(-(B^t)^{-1} Ax, -Bx +C (B^t)^{-1} Ax)\}.
\end{multline}
In order to have (\ref{eq1.12.2}), the matrix $A$ should necessarily be bijective, and we assume that this is the case. Writing $y = -(B^{t})^{-1} Ax$,
$x = -A^{-1} B^t y$, we then get from (\ref{eq1.12.4}),
\begin{multline*}
\kappa_{\varphi}^{-1}(\{(x,0);\,x\in \comp^n\}) = \{(y, B A^{-1} B^t y - C (B^{t})^{-1} AA^{-1} B^t y)\} \\
= \{(y, \left(B A^{-1} B^t - C\right)y)\}.
\end{multline*}
The condition (\ref{eq1.12.2}) therefore holds precisely when
\begeq
\label{eq1.12.5}
B A^{-1} B^t - C = F_+.
\endeq
Using (\ref{eq1.12.31}), we may rewrite (\ref{eq1.12.5}) in the form
$$
B A^{-1} B^t = F_+ - F_-,
$$
and observe that the matrix $F_+ - F_-$ is invertible, since ${\rm Im}\,\left(F_+ - F_-\right) > 0$. It follows that $A^{-1} = B^{-1}(F_+ - F_-)(B^{t})^{-1}$, and
choosing the invertible symmetric matrix $A$ in the form
$$
A = B^t \left(F_+ - F_-\right)^{-1} B,
$$
we achieve (\ref{eq1.12.2}). The general solution to (\ref{eq1.12.2}), (\ref{eq1.12.3}), satisfying (\ref{eq1.12.1}), is therefore of the form
$$
\varphi(x,y) = \frac{1}{2} B^t \left(F_+ - F_-\right)^{-1} Bx\cdot x + Bx\cdot y - \frac{1}{2}F_-y\cdot y.
$$
Here $B$ is an arbitrary invertible matrix.

\section{Metaplectic FBI transforms and Bergman\\ kernels}
\setcounter{equation}{0}
Last time we discussed the geometry of complex Lagrangian planes in the complexified phase space and that motivated us to look at complex canonical
transformations of the form
$$
\kappa_{\varphi}: \comp^{2n}\ni (y,-\varphi'_y(x,y)) \mapsto (x,\varphi'_x(x,y))\in \comp^{2n}.
$$
Here $\varphi$ is a holomorphic quadratic form on $\comp^n_x \times \comp^n_y$ such that
\begeq
\label{eq2.1}
{\rm det}\, \varphi''_{xy} \neq 0,\quad \operatorname{Im\,}\varphi''_{yy} > 0.
\endeq

\begin{dref}
The metaplectic Fourier-Bros-Iagolnitzer (FBI) transform associated to the quadratic form $\varphi$ satisfying {\rm (\ref{eq2.1})} is the operator
\begeq
\label{eq2.2}
T: \mathcal{S}'(\real^n)\to \operatorname{Hol}(\comp^n),
\endeq
given by
\begeq
\label{eq2.3}
Tu(x;h) = Ch^{-\frac{3n}{4}}\int e^{i\varphi(x,y)/h}u(y)\, dy,\quad 0 < h\leq 1.
\endeq
\end{dref}

\medskip
\noindent
To understand the growth properties of the entire function $Tu$ in the complex domain, let us set
\begeq
\label{eq2.31}
\Phi(x) = \sup_{y\in \mathbf{R}^n} (-\operatorname{Im\,}\varphi(x,y)).
\endeq
Since ${\rm Im}\, \varphi''_{yy} > 0$, we see that the supremum in (\ref{eq2.31}) is achieved at a unique point $y(x)\in \real^n$, which is the unique
critical point of the function
$$
\real^n \ni y\mapsto -{\rm Im}\, \varphi(x,y).
$$
It follows that
\begeq
\label{eq2.4}
\Phi(x) = {\rm vc}_{y\in {\mathbf R}^n} \left(-{\rm Im}\, \varphi(x,y)\right) = -{\rm Im}\,\varphi(x,y(x)),
\endeq
and by Taylor's formula, we can write, for $y\in \real^n$,
$$
-{\rm Im}\, \varphi(x,y) = \Phi(x) - \frac{1}{2} {\rm Im}\,\varphi''_{yy}(y-y(x))\cdot (y-y(x)) \leq \Phi(x) - \frac{1}{C}\abs{y-y(x)}^2.
$$
It is therefore clear that for some $M>0$ depending on the order of the distribution $u$, we have
\begeq
|Tu(x;h)| \leq Ch^{-M}\left\langle x \right\rangle^M e^{\Phi(x)/h},\quad x\in \comp^n.
\endeq
We also observe that the quadratic form $\Phi(x) = \sup_{y\in {\mathbf R}^n}\left(-\operatorname{Im\,}\varphi(x,y)\right)$ is pluri-subharmonic,
being the supremum of a family of pluri-harmonic quadratic forms.

\medskip
\noindent
{\bf Example}. Let $\varphi(x,y) = \frac{i}{2}(x-y)^2$. Then $\Phi(x) = \frac{1}{2}(\operatorname{Im\,}x)^2$, and the canonical transformation
$\kappa_{\varphi}$ is given by
$$
\kappa_{\varphi}(y,\eta) = (y-i\eta,\eta).
$$

\medskip
\noindent
{\it Remark}. In microlocal analysis, microlocal properties of $u\in {\cal S}'(\real^n)$ near $(y,\eta)\in T^*\real^n\backslash\{0\}$ can be characterized using local properties of the
holomorphic function $Tu$ near $\pi_x\left(\kappa_{\varphi}(y,\eta)\right) \in \comp^n$. Here $\pi_x: \comp_{x,\xi}^{2n}\ni (x,\xi) \rightarrow x\in \comp^n$ is the
natural projection map. We refer to~\cite{SjAst} and to Section 2.6 of Chapter 2 of this text for further details. In this elementary discussion, we shall only be
concerned with global aspects of the metaplectic FBI transforms.

\bigskip
\noindent
The following proposition indicates that there is a dictionary between the real side and the FBI tranform side, where $\real^{2n}$ corresponds to the linear
manifold
\begeq
\label{eq2.5}
\Lambda_\Phi = \left\{ \left(x,\frac{2}{i} \frac{\partial \Phi}{\partial x}(x) \right);\,\, x\in \comp^n \right\} \subseteq \comp^{2n}.
\endeq

\begin{prop}
The complex canonical transformation
\begeq
\kappa_{\varphi}: \comp^{2n} \ni (y,-\varphi'_y(x,y)) \mapsto (x,\varphi'_x(x,y))\in \comp^{2n}
\endeq
maps $\real^{2n}$ bijectively onto $\Lambda_\Phi$. The quadratic form $\Phi$ introduced in {\rm (\ref{eq2.31})} is strictly pluri-subharmonic.
\end{prop}
\begin{proof}
We claim that for any $x\in \comp^n$ there is a unique $(y(x),\eta(x))\in \real^{2n}$ such that $\pi_x \circ \kappa_\varphi(y(x),\eta(x)) = x$. Indeed, if
$y\in\real^n$, then $\varphi'_y(x,y)$ is real if and only if $\nabla_y(-{\rm Im}\,\varphi(x,y))= 0$, in other words, if and only if $y = y(x)$, the
critical point in (\ref{eq2.4}). The claim follows with $\eta(x) = -\varphi'_y(x,y(x))$. We let next $\xi(x)\in \comp^n$ be such that
$\kappa_\varphi(y(x),\eta(x)) = (x,\xi(x))$, i.e. $\xi(x) = \varphi'_x(x,y(x))$. Writing
$$
\Phi(x) = -{\rm Im}\, \varphi(x,y(x)) = \frac{i}{2}\left(\varphi(x,y(x)) - \overline{\varphi(x,y(x))}\right),
$$
we check, using the fact that $\varphi'_y(x,y(x))$ and $y(x)$ are real that
\begeq
\xi(x) = \frac{2}{i} \frac{\partial \Phi}{\partial x}(x).
\endeq
It follows that $\kappa_{\varphi}(\real^{2n}) = \Lambda_{\Phi}$, and since $\sigma|_{\mathbf{R}^{2n}}$ is non-degenerate, we obtain that $\sigma|_{\Lambda_\Phi}$
is non-degenerate, or equivalently, the Levi form $\partial_{\bar x}\partial_x\Phi $ is non-degenerate. Since we already know that $\Phi$ is pluri-subharmonic,
we conclude that $\Phi$ is strictly pluri-subharmonic.
\end{proof}

\bigskip
\noindent
We shall now establish the following basic result, concerning the mapping properties of the FBI transform on $L^2(\real^n)$.

\begin{theo}
If $C > 0$ is suitably chosen in {\rm (\ref{eq2.3})}, then $T$ is unitary,
$$
T: L^2(\mathbf{R}^n) \rightarrow H_\Phi(\mathbf{C}^n):= L^2(\mathbf{C}^n,e^{-2\Phi/h}\, L(dx))\cap \operatorname{Hol}(\mathbf{C}^n).
$$
Here $L(dx)$ is the Lebesgue measure on $\comp^n$.
\end{theo}

\medskip
\noindent
As a preparation for the proof, let us first derive an expression for the orthogonal (Bergman) projection:
$$
\Pi: L^2_\Phi(\comp^n) \to H_\Phi(\comp^n),
$$
where $L^2_\Phi(\comp^n) = L^2(\mathbf{C}^n,e^{-2\Phi/h}\, L(dx))$ and $H_{\Phi}(\comp^n) \subseteq L^2_{\Phi}(\comp^n)$ is the closed subspace of holomorphic
functions. Let $\psi(x,y)$ be the unique holomorphic quadratic form on $\mathbf{C}^n_x\times \mathbf{C}^n_y$ such that $\psi(x,\bar x) = \Phi(x)$. Here we may notice
that the anti-diagonal $\{(x,\bar x); x\in \comp^n\}$ is maximally totally real $\subseteq \comp^n_x \times \comp^n_y$. Explicitly, we have
$$
\psi(x,y) = \frac{1}{2}\Phi''_{xx}x\cdot x + \Phi_{\bar x x}x\cdot y + \frac{1}{2}\Phi''_{\bar x \bar x}y\cdot y,
$$
so that in particular, $\psi''_{xy} = \Phi''_{x\bar x}$ is non-degenerate. It also follows that when $y = \bar x$, we have
\begeq
\label{eq2.51}
\partial_y \psi = \partial_{\bar x} \Phi, \quad \partial_x \psi = \partial_x \Phi.
\endeq
These observations have the following useful consequence:
\begeq
\label{eq2.6}
2{\rm Re}\, \psi(x,\overline{y})-\Phi(x) - \Phi(y) = -\Phi''_{\bar x x}(y-x)\cdot(\overline{y-x}) \sim - \abs{y-x}^2,
\endeq
on $\comp^n_x \times \comp^n_y$. Here the last conclusion follows since $\Phi$ is strictly pluri-subharmonic, and to verify the first equality in (\ref{eq2.6})
it suffices to Taylor expand the quadratic functions $y\mapsto \Phi(y)$ and $y\mapsto \psi(x,\bar y)$ at the point $y=x$, and exploit (\ref{eq2.51}) to obtain some
cancelations.

\begin{prop}
The orthogonal projection $\Pi: L^2_\Phi(\comp^n)  \to H_\Phi(\comp^n)$ is given by
\begeq
\label{eq2.61}
\Pi u(x) = \frac{2^n \det\psi''_{xy}}{(\pi h)^n} \int_{\mathbf{C}^n} e^{2\psi(x,\bar y)/h}u(y) e^{-2\Phi(y)/h}\, L(dy).
\endeq
\end{prop}
\begin{proof}
Let $\Pi$ be the operator given in (\ref{eq2.61}). To see that
\begeq
\label{eq2.7}
\Pi = {\cal O}(1): L^2_{\Phi}(\comp^n) \rightarrow H_{\Phi}(\comp^n),
\endeq
we consider the reduced kernel
\begeq
\widetilde \Pi (x,y) = e^{-\Phi(x)/h} \Pi(x,y)e^{\Phi(y)/h},
\endeq
and observe that thanks to (\ref{eq2.6}), we have
$$
\abs{\widetilde \Pi(x,y)} \leq \frac{C}{h^n}e^{-|x-y|^2/Ch}.
$$
The uniform boundedness of $\Pi$ on $L^2_{\Phi}$ is therefore a consequence of Schur's lemma, and since the range of $\Pi$ consists of holomorphic functions, the
property (\ref{eq2.7}) follows. The selfadjointness of $\Pi$ on $L^2_{\Phi}$ follows since $\overline{\psi(x, \bar y)} = \psi(y,\bar x)$. We finally need to
show the reproducing property of $\Pi$,
\begeq
\label{eq2.8}
\Pi u = u, \quad u \in H_\Phi(\comp^n).
\endeq
To see (\ref{eq2.8}), we start by establishing the Fourier inversion formula in the complex domain,
\begeq
\label{eq2.9}
u(x) = \frac{1}{(2\pi h)^n} \iint_{\Gamma(x)} e^{\frac{i}{h}(x-y)\cdot \theta}u(y)\, dy\wedge d\theta,\quad u\in {H}_\Phi(\comp^n).
\endeq
Here $dy\wedge d\theta$ is a $(2n,0)$--form in $\comp^n_y \times \comp^n_{\theta}$, and the integration in (\ref{eq2.9}) is carried out over the $2n$-dimensional
contour (chain) $\Gamma(x)$, parametrized by $y\in \comp^n$ and given by
\begeq
\label{eq2.10}
\Gamma(x): \comp^n \ni y \mapsto \left(y,\theta\right)\in \comp^n \times \comp^n, \quad
\theta= \frac{2}{i} \frac{\partial \Phi}{\partial x}(x) + iC\overline{(x-y)}.
\endeq
Here $C\gg 1$ is large enough. We have
\begeq
\label{eq2.10.1}
dy\wedge d\theta|_{\Gamma(x)} = \left(\frac{C}{i}\right)^n dy \wedge d\bar y
\endeq
is real and non-vanishing, and it what follows we shall tacitly assume that the orientation on $\Gamma(x)$ has been chosen so that the form in
(\ref{eq2.10.1}) is a positive multiple of the Lebesgue measure on $\comp^n_y$. Let us also notice that the unique critical point of the function
$\comp^n \times \comp^n \ni (y,\theta) \mapsto -{\rm Im}\, (x-y)\cdot \theta + \Phi(y)$ is given by $y=x$,
$\theta = \frac{2}{i}\frac{\partial \Phi}{\partial x}(x)$, and the contour $\Gamma(x)$ passes through the critical point for all $C$. To see (\ref{eq2.9}), we
first observe that the contour $\Gamma(x)$ is good~\cite{Sj82}, since along $\Gamma(x)$, we have in view of Taylor's formula,
$$
\operatorname{Re\,}(i(x-y)\cdot \theta) + \Phi(y) - \Phi(x) \leq -\abs{x-y}^2,
$$
provided that $C > 1$ is large enough. The integral in (\ref{eq2.9}) therefore converges absolutely for all $u\in {\rm Hol}(\comp^n)$ such that
$\abs{u(x)} \leq {\cal O}_h(1) \langle{x\rangle}^{N_0} e^{\Phi(x)/h}$, for some $N_0 > 0$, and in particular, for all $u\in H_{\Phi}$. We also notice
that it is independent of $C \gg 1$, in view of Stokes' formula. 

\color{black}
\medskip
\noindent
Using (\ref{eq2.10}), we see that the right hand side in (\ref{eq2.9}) is given by
\begeq
\label{eq2.11}
\frac{2^n C^n}{(2\pi h)^n} \int e^{-C\abs{x-y}^2/h} e^{\frac{2}{h}\frac{\partial \Phi}{\partial x}(x)\cdot (x-y)} u(y)\, L(dy).
\endeq
Here the Gaussian
$$
\comp^n \ni y \mapsto \frac{C^n}{(\pi h)^n} e^{-C\abs{y}^2/h}
$$
is spherically symmetric of integral one, and therefore, by the mean value theorem for holomorphic functions, here applied to the function
$$
y \mapsto e^{\frac{2}{h}\frac{\partial \Phi}{\partial x}(x)\cdot (x-y)} u(y),
$$
we conclude that the expression (\ref{eq2.11}) is equal to $u(x)$ --- see also Lemma 7.3.11 in~\cite{HormI}. This establishes the validity of (\ref{eq2.9}), and
we may observe that the argument given above is in some sense simpler than the usual proof of Fourier's inversion formula in the real domain, since all the
integrals involved converge absolutely, thanks to the choice of a family of good contours, such as $\Gamma(x)$ above.

\bigskip
\noindent
We shall now finish the proof of Proposition 1.3.4 by passing from (\ref{eq2.9}) to (\ref{eq2.61}). To this end, we make a linear complex change of variables
$\theta \mapsto w$, given by
$$
\theta =  \frac{2}{i}\frac{\partial \psi}{\partial x}\left(\frac{x+y}{2},w\right) =
\frac{2}{i}\left(\Phi''_{xx}\left(\frac{x+y}{2}\right) + \Phi''_{x\bar x}w \right).
$$
It follows, since $\psi$ is quadratic, that
$$
2\left(\psi(x,w) - \psi(y,w)\right) = i(x-y)\cdot \theta,
$$
and we get therefore from (\ref{eq2.9}),
\begeq
\label{eq2.12}
u(x) = \frac{1}{(2\pi h)^n} \iint_{\widetilde{\Gamma}(x)} e^{\frac{2}{h}(\psi(x,w)-\psi(y,w))} \left(\frac{2}{i}\right)^n
\left({\rm det}\, \Phi_{x\bar x}\right) u(y)\, dy\wedge dw.
\endeq
Here $\widetilde{\Gamma}(x)$ is the natural image of $\Gamma(x)$, so that $(y,w)\in \widetilde{\Gamma}(x)$ precisely when $(y,\theta) \in \Gamma(x)$. The contour
$\widetilde{\Gamma}(x)$ is good in the sense that along $\widetilde{\Gamma}(x)$, we have
$$
2\operatorname{Re\,}(\psi(x,w) - \psi(y,w)) + \Phi(y) - \Phi(x) \leq -\abs{x-y}^2,
$$
and another good contour $\widehat{\Gamma}(x)$ is given by $w = \bar y$. Indeed, we have in view of (\ref{eq2.6}),
$$
2\operatorname{Re\,}(\psi(x,\bar y) - \psi(y,\bar y)) + \Phi(y) - \Phi(x) \leq -\frac{1}{C}\abs{x-y}^2.
$$
The good contour $\widehat{\Gamma}(x)$ is homotopic to $\widetilde{\Gamma}(x)$, with the homotopy being within the set of good contours, and
we conclude, in view of Stokes' formula, that
\begeq
\label{eq2.13}
u(x) = \frac{{\rm det}\, \Phi_{x\bar x}} {i^n(\pi h)^n} \iint_{\widehat{\Gamma}(x)} e^{\frac{2}{h}(\psi(x,w)-\psi(y,w))} u(y)\, dy\wedge dw = \Pi u.
\endeq
This completes the proof of Proposition 1.3.4.
\end{proof}

\bigskip
\noindent
We shall return to the proof of Theorem 1.3.3, where, without loss of generality, we may assume that
$$
\varphi''_{x x} = {\rm Re}\, \varphi''_{y y} = 0,
$$
so that we can write
\begeq
\label{eq2.14}
\varphi(x,y) = Ax \cdot y + \frac{i}{2} By \cdot y, \quad B > 0, \quad {\rm det}\, A \neq 0.
\endeq
We shall first show that $T: L^2(\real^n) \rightarrow H_{\Phi}(\comp^n)$ is an isometry. To this end, we observe that $Tu(A^{-1}x;h)$ is equal to
$Ch^{-3n/4}$ times the semiclassical Fourier-Laplace transform of $u(y)e^{-By\cdot y/2h}$, and therefore, by Parseval's formula,
$$
\int\abs{Tu(A^{-1}x;h)}^2\, d{\rm Re}\, x = (2\pi h)^n C^2 h^{-3n/2} \int e^{-By\cdot y/h} e^{-2{\rm Im}\,x \cdot y/h}\abs{u(y)}^2\,dy.
$$
Next, a computation using (\ref{eq2.14}) shows that
\begeq
\label{eq2.14.01}
\Phi(x) = \frac{1}{2} B^{-1} {\rm Im}\, (Ax)\cdot {\rm Im}\, (Ax),
\endeq
and therefore
\begin{multline*}
\iint \abs{Tu(A^{-1}x;h)}^2 e^{-2\Phi(A^{-1}x)/h}\, L(dx) \\
= (2\pi)^n C^2 h^{-n/2} \iint e^{-(By\cdot y +2\xi \cdot y + B^{-1}\xi \cdot \xi)/h} \abs{u(y)}^2\,dy\, d\xi.
\end{multline*}
We have $By\cdot y +2\xi \cdot y + B^{-1}\xi \cdot \xi = B^{-1}(\xi + By)\cdot (\xi + By)$, and therefore the integral with respect to $\xi$ in the right hand
side is equal to $(\pi h)^{n/2} \left({\rm det}\, B\right)^{1/2}$. On the other hand, the left hand side is given by $\abs{{\rm det}\, A}^2\norm{Tu}_{H_{\Phi}}^2$,
so that we get
$$
\abs{{\rm det}\, A}^2\norm{Tu}_{H_{\Phi}}^2 = 2^n \pi^{3n/2} C^2  \left({\rm det}\, B\right)^{1/2}\norm{u}_{L^2}^2.
$$
Choosing
\begeq
\label{eq2.14.1}
C = 2^{-n/2} \pi^{-3n/4} \left({\rm det}\, B\right)^{-1/4}\abs{{\rm det}\, A} > 0,
\endeq
we conclude that $T: L^2(\real^n) \rightarrow H_{\Phi}(\comp^n)$ is an isometry.

\medskip
\noindent
We shall finally show that $T T^* = 1$ on $H_{\Phi}(\comp^n)$. Here the Hilbert space adjoint $T^*$ of $T: L^2(\real^n) \rightarrow L^2_{\Phi}(\comp^n)$ is given by
\begeq
\label{eq2.15}
T^*v(y) = C h^{-3n/4} \int e^{-i\varphi^*(\bar x, y)/h} v(x) e^{-2\Phi(x)/h}\, L(dx),
\endeq
where $\varphi^*(x,y) = \overline{\varphi(\bar x, \bar y)}$ is the holomorphic extension of $\real^n_x \times \real^n_y \ni (x,y) \mapsto \overline{\varphi(x,y)}$.
We get, for $v\in {\rm Hol}(\comp^n)$, such that $\abs{v(x)} \leq {\cal O}_{N,h}(1) \langle{x\rangle}^{-N} e^{\Phi(x)/h}$, for all $N$,
\begeq
\label{eq2.15.01}
(T T^*v)(x) = C^2 h^{-3n/2} \iint e^{i(\varphi(x,y)-\varphi^*(\bar w, y))/h} v(w) e^{-2\Phi(w)/h}\, L(dw)\, dy.
\endeq
The integral with respect to $y$ can be computed by exact stationary phase and we get, writing $q(x,\bar w, y) = \varphi(x,y) - \varphi^*(\bar w, y)$,
\begeq
\label{eq2.15.1}
\int e^{i q(x,\bar w, y)/h}\, dy = h^{n/2} \left({\rm det}\, \frac{q''_{yy}}{2\pi i}\right)^{-1/2} e^{i{\rm vc}_y q(x,\bar w, y)/h}.
\endeq
Here
\begeq
\label{eq2.15.2}
\frac{i}{2} {\rm vc}_y (q(x,z,y)) = \frac{i}{2}{\rm vc}_y\left(\varphi(x,y) - \varphi^*(z,y)\right)
\endeq
is a holomorphic quadratic form on $\comp^n_x \times \comp^n_z$, and when $z = \bar x$, we see using (\ref{eq2.14}) that the unique critical point $y$ in
(\ref{eq2.15.2}) is real and that (\ref{eq2.15.2}) is equal to $\Phi(x)$. It follows that
$$
\frac{i}{2}{\rm vc}_y\left(\varphi(x,y) - \varphi^*(z,y)\right) = \psi(x,z),
$$
and using also that $q''_{yy} = 2iB$, we obtain from (\ref{eq2.15.1}) that
$$
\int e^{i q(x,\bar w, y)/h}\, dy = h^{n/2} \pi^{n/2} ({\rm det}\, B)^{-1/2} e^{2\psi(x,\bar w)/h}.
$$
Returning to (\ref{eq2.15.01}) and recalling the explicit expression for the constant $C$ in (\ref{eq2.14.1}), we see that
\begin{multline*}
(T T^* v)(x) = C^2 h^{-3n/2} h^{n/2} \pi^{n/2} ({\rm det}\, B)^{-1/2} \int e^{2\psi(x,\bar w)/h} v(w) e^{-2\Phi(w)/h}\, L(dw) \\
= \frac{2^{-n} ({\rm det}\, B)^{-1} \abs{{\rm det}\, A}^2}{(\pi h)^n} \int e^{2\psi(x,\bar w)/h} v(w) e^{-2\Phi(w)/h}\, L(dw) = (\Pi v)(x) = v(x),
\end{multline*}
where the penultimate equality follows from Proposition 1.3.4. Here we have also used that
$$
{\rm det}\, \Phi''_{x \bar x} = 4^{-n} \abs{{\rm det}\, A}^2 ({\rm det}\, B)^{-1},
$$
in view of (\ref{eq2.14.01}). The proof of Theorem 1.3.3 is complete.

\section{Pseudodifferential operators on FBI\\ transform side}
\setcounter{equation}{0}
Let $\Phi$ be a strictly pluri-subharmonic quadratic form on $\comp^n$, and let us recall the linear IR-manifold $\Lambda_{\Phi}\subset \comp^n_x \times \comp^n_{\xi}$,
defined in (\ref{eq2.5}). Introduce
\begeq
\label{eq3.1}
S(\Lambda_{\Phi}) = \{a\in C^{\infty}(\Lambda_{\Phi});\,\partial^{\alpha} a = {\cal O}_{\alpha}(1),\,\,\forall \alpha\}
\endeq
Here we identify $\Lambda_{\Phi}$ linearly with $\comp^n$ via the projection map $\Lambda_{\Phi}\ni (x,\xi) \mapsto x\in \comp^n$. If $a\in S(\Lambda_{\Phi})$ and
$u\in {\rm Hol}(\comp^n)$ is such that $u = {\cal O}_{h,N}(1) \langle{x\rangle}^{-N}e^{\Phi(x)/h}$, for all $N\geq 0$, we put
\begeq
\label{eq3.2}
{\rm Op}_{h}^w(a)u(x) = \frac{1}{(2\pi h)^n}\int\!\!\!\!\int_{\Gamma(x)} e^{\frac{i}{h}(x-y)\cdot \theta} a\left(\frac{x+y}{2},\theta\right)u(y)\, dy\wedge d\theta.
\endeq
Here $\Gamma(x)$ is the only possible integration contour given by
$$
\theta = \frac{2}{i} \frac{\partial \Phi}{\partial x}\left(\frac{x+y}{2}\right).
$$
Along $\Gamma(x)$, we get, by Taylor's formula,
$$
{\rm Re}\left(i(x-y)\cdot \theta\right) - \Phi(x) + \Phi(y) = \biggl\langle{x-y,\nabla \Phi \left(\frac{x+y}{2}\right)\biggr\rangle}_{{\bf R}^{2n}} - \Phi(x) + \Phi(y) = 0,
$$
and let us notice also that
$$
dy\wedge d\theta|_{\Gamma(x)} = \frac{1}{i^n} {\rm det}\,(\Phi''_{x\bar{x}})dy\wedge d\bar{y}.
$$
It follows that the integral in (\ref{eq3.2}) converges absolutely, and for a suitable constant $C\neq 0$, we may write,
\begeq
\label{eq3.3}
{\rm Op}_{h}^w(a)u(x) = \frac{C}{h^n}\int K(x,y)u(y)\, L(dy),
\endeq
where
$$
K(x,y) = e^{\frac{2}{h}(x-y)\cdot \frac{\partial \Phi}{\partial x}\left(\frac{x+y}{2}\right)}
a\left(\frac{x+y}{2}, \frac{2}{i}\frac{\partial \Phi}{\partial x}\left(\frac{x+y}{2}\right)\right).
$$
It follows that $\partial_{\bar x} K(x,y) = \partial_{\bar y}K(x,y)$, and using an integration by parts we conclude that the function
${\rm Op}_h^w(a)u(x)$ is holomorphic, since $u$ is.

\medskip
\noindent
\begin{theo} Let $a\in S(\Lambda_{\Phi})$. The operator ${\rm Op}_h^w(a)$ extends to a bounded operator: $H_{\Phi}(\comp^n) \rightarrow H_{\Phi}(\comp^n)$, whose
norm is ${\cal O}(1)$, as $h\rightarrow 0^+$.
\end{theo}
\begin{proof}
Following~\cite{Sj02}, we shall prove this result by means of a contour deformation argument. When $0\leq t \leq 1$, let $\Gamma_t(x)$ be the $2n$-dimensional
contour, given by
\begeq
\label{eq3.4}
\theta = \frac{2}{i} \frac{\partial \Phi}{\partial x}\left(\frac{x+y}{2}\right) + it \frac{\overline{x-y}}{\langle{x-y\rangle}}.
\endeq
We also introduce the $(2n+1)$-dimensional contour $G(x) \subset \comp^n_y \times \comp^n_{\theta}$, given by
$$
G(x) = \bigcup_{0 \leq t\leq 1} \Gamma_t(x).
$$
We would like to replace the contour $\Gamma(x)=\Gamma_0(x)$ by $\Gamma_1(x)$ in (\ref{eq3.2}), and to that end, we let $\widetilde{a}\in C^{\infty}(\comp^{2n}_{x,\xi})$ be an almost
holomorphic extension of $a\in S(\Lambda_{\Phi})$, so that ${\rm supp}\,(\widetilde{a}) \subseteq \Lambda_{\Phi} + {\rm neigh}(0,\comp^{2n})$, all derivatives
of $\widetilde{a}$ are bounded, $\widetilde{a}|_{\Lambda_{\Phi}} = a$, and
\begeq
\label{eq3.41}
\abs{\partial_{\bar x,\bar \xi} \widetilde{a}(x,\xi)} \leq {\cal O}_N(1) \abs{\xi - \frac{2}{i}\frac{\partial \Phi}{\partial x}(x)}^N,
\endeq
for all $N \geq 0$. Let us recall that to construct $\widetilde{a}$, we may first make a complex linear change of coordinates to replace $\Lambda_{\Phi}$ by
$\real^{2n}$ and consider the problem of constructing an almost holomorphic extension of $a\in C^{\infty}(\real^{2n})$, with
$\partial^{\alpha} a \in L^{\infty}(\real^{2n})$ for all $\alpha$. To this end, following the classical construction by H\"ormander, explained in~\cite{DSj},
we set
\begeq
\label{eq3.5}
\widetilde{a}(X + iY) = \sum_{\abs{\alpha}\geq 0} \frac{\partial^{\alpha}a(X)}{\alpha!} (iY)^{\alpha} \chi(t_{\abs{\alpha}}Y),
\endeq
where $\chi\in C^{\infty}_0(\real^{2n})$, $\chi = 1$ near $0$, and $t_j \rightarrow \infty$ sufficiently rapidly. Returning to (\ref{eq3.2}), we get by Stokes'
formula, assuming that $u\in {\rm Hol}(\comp^n)$, with $u(x) = {\cal O}_{h,N}(1) \langle{x\rangle}^{-N}e^{\Phi(x)/h}$, for all $N\geq0$,
\begeq
\label{eq3.6}
{\rm Op}^w_h(a)u = I_1 u + I_2 u,
\endeq
where
\begeq
\label{eq3.7}
I_1 u(x) = \frac{1}{(2\pi h)^n}\int\!\!\!\!\int_{\Gamma_1(x)} e^{\frac{i}{h}(x-y)\cdot \theta} \widetilde{a}\left(\frac{x+y}{2},\theta\right)u(y)\, dy\wedge d\theta,
\endeq
and
\begeq
\label{eq3.8}
I_2 u(x) = \frac{1}{(2\pi h)^n} \int_{G(x)}
d_{y,\theta} \left(e^{\frac{i}{h}(x-y)\cdot\theta} \widetilde{a}\left(\frac{x+y}{2},\theta\right)u(y)\right)\wedge dy\wedge d\theta.
\endeq
We have $dy\wedge d\theta|_{\Gamma_1(x)} = {\cal O}(1) L(dy)$, and it follows from (\ref{eq3.4}) that the reduced kernel of $I_1$ satisfies
$$
\abs{e^{-\Phi(x)/h} I_1(x,y) e^{\Phi(y)/h}} \leq \frac{C}{h^n} e^{-\frac{\abs{x-y}^2}{h\langle{x-y\rangle}}}.
$$
In order to conclude that $I_1 = {\cal O}(1): L^2_{\Phi}(\comp^n) \rightarrow L^2_{\Phi}(\comp^n)$, in view of Schur's lemma, it suffices to check that
$$
\frac{1}{h^n} \int e^{-\frac{\abs{x}^2}{h\langle{x\rangle}}}\, L(dx) = {\cal O}(1),
$$
which is easily seen by considering the integrals over the regions where $\abs{x}\leq 1$ and $\abs{x}\geq 1$. When estimating the contribution of $I_2$, we write
\begin{multline*}
d_{y,\theta} \left(e^{\frac{i}{h}(x-y)\cdot\theta} \widetilde{a}\left(\frac{x+y}{2},\theta\right)u(y)\right)\wedge dy\wedge d\theta \\
= e^{\frac{i}{h}(x-y)\cdot\theta} u(y)\partial_{\bar y,\bar \theta} \left(\widetilde{a}\left(\frac{x+y}{2},\theta\right)\right)\wedge dy\wedge d\theta,
\end{multline*}
and notice that in view of (\ref{eq3.41}), we have along $G(x)$,
$$
\partial_{\bar y,\bar \theta} \left(\widetilde{a}\left(\frac{x+y}{2},\theta\right)\right)\wedge dy\wedge d\theta =
{\cal O}_N(1) t^N \frac{\abs{x-y}^N}{\langle{x-y\rangle}^N}dt\, L(dy),\quad N\geq 0.
$$
It follows that the reduced kernel of $I_2$ satisfies
$$
\abs{e^{-\Phi(x)/h} I_2(t,x,y)e^{\Phi(y)/h}} \leq \frac{C}{h^n} e^{-\frac{t\abs{x-y}^2}{h\langle{x-y\rangle}}}t^N \frac{\abs{x-y}^N}{\langle{x-y\rangle}^N},
$$
and by an application of Schur's lemma, we see that in order to control the norm of the operator
$$
I_2: L^2_{\Phi}(\comp^n) \rightarrow L^2_{\Phi}(\comp^n),
$$
it suffices to estimate
$$
\frac{1}{h^n} \int e^{-\frac{t\abs{x}^2}{h\langle{x\rangle}}} t^N \frac{\abs{x}^N}{\langle{x\rangle}^N}\, L(dx),
$$
uniformly in $t\in [0,1]$. In doing so, we consider first the contribution of the region where $\abs{x}\leq 1$. We get
\begin{multline*}
\frac{1}{h^n} \int_{\abs{x}\leq 1} e^{-\frac{t\abs{x}^2}{h\langle{x\rangle}}} t^N \frac{\abs{x}^N}{\langle{x\rangle}^N}\, L(dx) =
{\cal O}(1) h^{-n} \int_0^1 e^{-\frac{tr^2}{2h}}t^N r^{N+2n-1}dr \\
\leq {\cal O}(1) h^{-n} \int_0^{\infty} e^{-s^2} t^N \left(\frac{2h}{t}\right)^{N/2+n} s^{N+2n-1}\,ds = {\cal O}(1) h^{N/2} t^{N/2-n} = {\cal O}(h^{N/2}),
\end{multline*}
uniformly in $t\in [0,1]$, for $N$ large enough. Next, the contribution of the integral over the region $\abs{x}\geq 1$ does not exceed a constant times
\begin{multline*}
h^{-n} \int_{\abs{x}\geq 1} e^{-\frac{t\abs{x}}{2h}} t^N L(dx) = {\cal O}(1) h^{-n}\int_1^{\infty} e^{-\frac{tr}{2h}} t^N r^{2n-1}\,dr \\
= {\cal O}(1) h^n t^{N-2n} \int_{t/h}^{\infty} e^{-\rho/2} \rho^{2n-1}\, d\rho = {\cal O}(1) h^n t^{N-2n} {\cal O}\left(\left(1+\frac{t}{h}\right)^{-M}\right),
\end{multline*}
for all $M\geq 0$. If $t\leq h^{1/2}$, we use the factor $t^{N-2n}$ to get the bound ${\cal O}(h^{N/2})$, while for $t\geq h^{1/2}$, we use the factor
$$
{\cal O}\left(\left(1+\frac{t}{h}\right)^{-M}\right) = {\cal O}(h^{M/2}),
$$
to get the bound ${\cal O}(h^{n+M/2})$. We conclude, in view of (\ref{eq3.6}) that
\begeq
\label{eq3.9}
{\rm Op}_h^w(a)u(x) =
\frac{1}{(2\pi h)^n}\int\!\!\!\!\int_{\Gamma_1(x)} e^{\frac{i}{h}(x-y)\cdot \theta} \widetilde{a}\left(\frac{x+y}{2},\theta\right)u(y)\, dy\wedge d\theta
+ Ru,
\endeq
where
$$
R = {\cal O}(h^{\infty}): L^2_{\Phi}(\comp^n) \rightarrow L^2_{\Phi}(\comp^n).
$$
This completes the proof.
\end{proof}

\bigskip
\noindent
We shall next discuss the link between the $h$-pseudodifferential operators on the FBI transform side and the semiclassical Weyl quantization on $\real^n$. We have
the following metaplectic Egorov theorem.
\begin{theo}
Let $T: L^2(\real^n) \rightarrow H_{\Phi}(\comp^n)$ be a metaplectic FBI transform with the associated canonical transformation
$$
\kappa_T: \real^{2n} \rightarrow \Lambda_{\Phi}.
$$
If $a\in S(\Lambda_{\Phi})$ then we have
$$
T^* {\rm Op}_h^w(a) T = {\rm Op}_h^w(a\circ\kappa_T).
$$
Here the operator in the right hand side is the $h$-Weyl quantization of the symbol $a\circ \kappa_T \in S(1)$ on $\real^{n}$.
\end{theo}
\begin{proof}
The starting point is the following fact that can be verified by means of an explicit computation: let $\ell$ be a real linear form on $\real^{2n}$ and let
$k$ be the linear form on $\Lambda_{\Phi}$ such that $k\circ \kappa_T = \ell$. Then we have on ${\cal S}(\real^n)$,
\begeq
\label{eq3.10}
{\rm Op}_h^w(k) \circ T = T \circ {\rm Op}_h^w(l).
\endeq
In the computation, it is convenient to use that if $k(x,\xi) = x^*\cdot x + \xi^* \cdot \xi$, $x,\xi \in \comp^n$, then
$$
{\rm Op}^w_h(k) = k(x,hD_x) = x^*\cdot x + \xi^*\cdot hD_x,
$$
and there is a similar formula for ${\rm Op}^w_h(\ell)$. Now let us recall from~\cite{DSj} that the first order operator $\ell(x,hD_x) = {\rm Op}^w_h(\ell)$ is essentially selfadjoint on
$L^2(\real^n)$ from ${\cal S}(\real^n)$, and
\begeq
\label{eq3.10.1}
e^{i\ell(x,hD_x)/h} = {\rm Op}^w_h\left(e^{i\ell(x,\xi)/h}\right).
\endeq
It follows from (\ref{eq3.10}) and the unitarity of $T$ that $k(x,hD_x)$ is essentially selfadjoint on $H_{\Phi}(\comp^n)$ from $T{\cal S}(\real^n)$, and
therefore, the corresponding unitary groups are intertwined by $T$,
$$
e^{ik(x,hD_x)/h}\circ T = T \circ e^{il(x,hD_x)/h}.
$$
Here we claim that in analogy with (\ref{eq3.10.1}), we have
\begeq
\label{eq3.11}
e^{ik(x,hD_x)/h} = {\rm Op}_h^w(e^{ik(x,\xi)/h}),
\endeq
where the right hand side is still given by the contour integral in (\ref{eq3.2}). Indeed, let us write, for $u\in T{\cal S}(\real^n)$,
\begeq
\label{eq3.12}
{\rm Op}_h^w\left(e^{ik(x,\xi)/h}\right) u(x) = \frac{1}{(2\pi h)^n}\int\!\!\!\int_{\Gamma(x)}
e^{\frac{i}{h}\left((x-y+\xi^*)\cdot \theta + x^*\cdot\left(\frac{x+y}{2}\right)\right)}u(y)\, dy\wedge d\theta.
\endeq
Here by Stokes' theorem, the integration contour can be deformed to the following,
$$
\theta = \frac{2}{i}\frac{\partial \Phi}{\partial x}(x) + iC\overline{(x-y+\xi^*)},
$$
for $C\gg 1$ large enough, and the expression (\ref{eq3.12}) becomes
$$
\frac{2^n C^n}{(2\pi h)^n} \int e^{-C\abs{x-y+\xi^*}^2/h} e^{\frac{2}{h}(x-y+\xi^*)\cdot \frac{\partial \Phi}{\partial x}(x) +
\frac{i}{h}x^*\cdot \left(\frac{x+y}{2}\right)} u(y)\, L(dy),
$$
which, by the mean value theorem for holomorphic functions, is equal to
$$
x\mapsto e^{\frac{i}{h}x^*\cdot x} e^{\frac{i}{2h}x^*\cdot \xi^*} u(x+\xi^*) = e^{ik(x,hD_x)/h}u(x).
$$
This establishes (\ref{eq3.11}) and therefore, we get
\begeq
\label{eq3.13}
{\rm Op}_h^w\left(e^{\frac{i}{h}k(x,\xi)}\right)\circ T = T \circ {\rm Op}_h^w\left(e^{\frac{i}{h}\ell(x,\xi)}\right).
\endeq
If $a\in {\cal S}(\Lambda_{\Phi})$ and $b\in {\cal S}(\real^{2n})$ are related by $b=a\circ \kappa_T$, then by Fourier's inversion formula, we can represent
$a$ and $b$ as superpositions of bounded exponentials of the form $e^{ik(x,\xi)/h}$ and $e^{il(x,\xi)/h}$, respectively. Here the linear forms $k$ and $\ell$ are
related by $\ell = k \circ \kappa_T$, and passing to the $h$--Weyl quantizations, we get, in view of (\ref{eq3.13}),
\begeq
\label{eq3.14}
{\rm Op}^w_h (a) \circ T = T\circ {\rm Op}^w_h(b).
\endeq
A density argument allows us to complete the proof.
\end{proof}

\bigskip
\noindent
We shall finally make some remarks concerning pseudodifferential operators with holomorphic symbols, referring to~\cite{SjAst}, as well as to the second part of this
text, for a much more extensive discussion. Let us assume that $a(x,\xi)$ is a holomorphic bounded function in a region of the form
$\Lambda_{\Phi} + W \subset \comp^n_x \times \comp^n_{\xi}$. Here $W$ is a bounded open neighborhood of $0\in \comp^{2n}$. It follows from the proof of Theorem
1.4.1 that in this case we have, for $u\in H_{\Phi}(\comp^n)$,
\begeq
\label{eq3.15}
{\rm Op}_h^w(a)u(x) = \frac{1}{(2\pi h)^n}\int\!\!\!\int_{\Gamma_C(x)} e^{\frac{i}{h}(x-y)\cdot \theta} a\left(\frac{x+y}{2},\theta\right)u(y)\,dy\wedge d\theta,
\endeq
where the contour $\Gamma_C(x)$ is given by
$$
\theta = \frac{2}{i}\frac{\partial \Phi}{\partial x}\left(\frac{x+y}{2}\right) + \frac{i}{C}\frac{\overline{(x-y)}}{\langle{x-y\rangle}},
$$
and $C>0$ is large enough fixed, so that $\Gamma_C(x) \subset \Lambda_{\Phi} + W$. The holomorphy of the symbol allows us to consider weight functions different
from $\Phi$ as well, and study boundedness properties of ${\rm Op}_h^w(a)$ in the corresponding exponentially weighted spaces.

\medskip
\noindent
Following~\cite{Sj02}, we have the following result.
\begin{theo}
Let $\widetilde \Phi \in C^{1,1}(\comp^n)$ be such that $\widetilde \Phi(x) = \Phi(x) + f(x)$, where $f\in C^{1,1}_0(\comp^n)$ is such that
$\norm{\nabla f}_{L^\infty}$, $\norm{\nabla^2 f}_{L^\infty}$ are sufficiently small. We then have a uniformly bounded operator
\begeq
\label{eq3.16}
{\rm Op}_{h}^w(a) = {\cal O}(1): H_{\widetilde \Phi}(\comp^n) \to H_{\widetilde \Phi}(\comp^n).
\endeq
Here we set $H_{\widetilde \Phi}(\comp^n) = {\rm Hol}(\comp^n) \cap L^2(\comp^n, e^{-2\widetilde{\Phi}/h}\, L(dx))$.
\end{theo}
\begin{proof}
We make a deformation to the new contour and set
\begeq
\label{eq3.17}
{\rm Op}_h^w(a)u(x) = \frac{1}{(2\pi h)^n}\int\!\!\!\int_{\widetilde{\Gamma}_C(x)}
e^{\frac{i}{h}(x-y)\cdot \theta} a\left(\frac{x+y}{2},\theta\right)u(y)\,dy\wedge d\theta,
\endeq
where
\begeq
\label{eq3.18}
\widetilde \Gamma_C(x) = \frac{2}{i}\frac{\partial \widetilde \Phi}{\partial x}\left(\frac{x+y}{2} \right) +
\frac{i}{C}\frac{\overline{x-y}}{\left\langle x-y \right\rangle}.
\endeq
Along the contour $\widetilde{\Gamma}_C(x)$, we have
\begin{multline*}
-\widetilde{\Phi}(x) + {\rm Re}\, \left(i(x-y)\cdot \theta\right) + \widetilde{\Phi}(y) \\
= -\widetilde{\Phi}(x) + \biggl\langle{x-y,\nabla \widetilde{\Phi}\left(\frac{x+y}{2}\right)\biggr\rangle}_{{\bf R}^{2n}} + \widetilde{\Phi}(y) -
\frac{1}{C} \frac{\abs{x-y}^2}{\langle{x-y\rangle}} \\
= -f(x) + \biggl\langle{x-y,\nabla f\left(\frac{x+y}{2}\right)\biggr\rangle}_{{\bf R}^{2n}} + f(y) -
\frac{1}{C} \frac{\abs{x-y}^2}{\langle{x-y\rangle}},
\end{multline*}
and applying Taylor's formula we see that this expression does not exceed
$$
{\cal O}(1)\norm{f''}_{L^{\infty}} \frac{\abs{x-y}^2}{\langle{x-y\rangle}} - \frac{1}{C} \frac{\abs{x-y}^2}{\langle{x-y\rangle}} \leq
-\frac{1}{2C}\frac{\abs{x-y}^2}{\langle{x-y\rangle}},
$$
provided that $\norm{f''}_{L^{\infty}}$ is small enough. The proof can therefore be concluded as before, by an application of Schur's lemma.
\end{proof}

\medskip
\noindent
{\it Remark}. Let us notice that $H_{\widetilde{\Phi}}(\comp^n) = H_{\Phi}(\comp^n)$ as linear spaces, with the norms being equivalent, but not uniformly as
$h\rightarrow 0^+$. We observe also that the Lipschitz IR-manifold $\Lambda_{\widetilde{\Phi}}$ is close to $\Lambda_{\Phi}$, in the sense of Lipschitz graphs.

\bigskip
\noindent
It turns out that the natural symbol associated to the operator in (\ref{eq3.16}) is $a|_{\Lambda_{\widetilde{\Phi}}}$. Indeed, we have the following fundamental
quantization-multiplication formula, due to~\cite{Sj90},~\cite{CF}.

\begin{prop}
We have
\begin{align*}
\left({\rm Op}_{h}^w(a) u,v \right)_{H_{\widetilde \Phi}} =
\int a\left(x,\frac{2}{i}\frac{\partial \widetilde \Phi}{\partial x}(x) \right)u(x)\overline{v(x)}e^{-\frac{2}{h}\widetilde \Phi(x)}\, L(dx) +
\mathcal{O}(h)\norm{u}_{H_{\widetilde{\Phi}}} \,\norm{v}_{H_{\widetilde{\Phi}}},
\end{align*}
for $u,v\in H_{\widetilde \Phi}(\comp^n)$.
\end{prop}
\begin{proof}
We represent the operator ${\rm Op}_h^w(a)$ as in (\ref{eq3.17}) with the contour (\ref{eq3.18}), and Taylor expand $a$, writing
$\xi(x) = \frac{2}{i}\frac{\partial \widetilde \Phi}{\partial x}(x)$,
\begin{multline*}
a\left(\frac{x+y}{2},\theta \right) = a(x,\xi(x)) + (\partial_\xi a)(x,\xi(x))(\theta-\xi(x)) \\ +
(\partial_x a)(x,\xi(x))\left(\frac{y-x}{2} \right) + \mathcal{O}(|y-x|^2) + \mathcal{O}(|\theta - \xi(x)|^2).
\end{multline*}
Here the remainder terms are both ${\cal O}(\abs{x-y}^2)$ along the contour $\widetilde{\Gamma}_C(x)$, and therefore, in view of Schur's lemma,
their contribution gives rise to an operator of the norm ${\cal O}(h): H_{\widetilde{\Phi}}(\comp^n) \to L^2_{\widetilde{\Phi}}(\comp^n)$. Next, observing
that the term $(\partial_x a)(x,\xi(x))\left(\frac{y-x}{2} \right)$ drops out, when passing to the quantizations, we conclude that
$$
{\rm Op}_h^w(a) = a(x,\xi(x)) + (\partial_\xi a)(x,\xi(x))\cdot(hD_x-\xi(x)) + R,
$$
where
$$
R = {\cal O}(h): H_{\widetilde{\Phi}}(\comp^n) \to L^2_{\widetilde{\Phi}}(\comp^n).
$$
It remains to estimate the integral
\begeq
\label{eq3.19}
\int \left(\partial_{\xi_j}a\right)(x,\xi(x))\left(\left(hD_{x_j} - \xi_j(x)\right)u(x)\right)\, \overline{v(x)} e^{-2\widetilde{\Phi}(x)/h}\, L(dx),
\quad 1\leq j\leq n,
\endeq
and since the function $\left(\partial_{\xi_j}a\right)(x,\xi(x))$ is Lipschitz, we can integrate by parts in (\ref{eq3.19}), getting
${\cal O}(h)\norm{u}_{H_{\widetilde{\Phi}}} \,\norm{v}_{H_{\widetilde{\Phi}}}$ plus the term
$$
\int \left(\partial_{\xi_j}a\right)(x,\xi(x)) u(x)\overline{v(x)} \left(-hD_{x_j} - \xi_j(x)\right)e^{-2\widetilde{\Phi}(x)/h}\, L(dx) = 0.
$$
This completes the proof.
\end{proof}

\bigskip
\noindent
We shall finish with the following general idea suggested by the discussion above: given an $h$--pseudo\-diffe\-ren\-tial operator of the form ${\rm Op}_h^w(a)$, with $a$
holomorphic in a tubular neighborhood of $\Lambda_{\Phi}$, try to find an IR-manifold $\Lambda_{\widetilde \Phi}$ close to $\Lambda_\Phi$ so that the operator
$$
{\rm Op}_h^w(a): H_{\widetilde{\Phi}}(\comp^n) \to H_{\widetilde{\Phi}}(\comp^n)
$$
acquires some improved properties, such as the invertibility, ellipticity, normality, etc. We refer to the works \cite{DeSjZw}, \cite{Hi04}, \cite{HiSj1},
\cite{HiSjVu07}, \cite{HiSj15}, \cite{MeSj1}, \cite{MeSj2}, where implementations of this idea have led to some precise results in the spectral theory of
semiclassical non-selfadjoint operators. It may also be interesting to compare this idea with the recent developments around Carleman estimates with
limiting Carleman weights for second order elliptic differential operators, see~\cite{KSjU}.

\chapter{Analytic microlocal analysis using holomorphic functions with exponential weights}

\section{Introduction}\label{}
\setcounter{equation}{0}
There are several approches to analytic microlocal analysis:
\begin{itemize}
\item One very natural approach consists in adapting the classical
  theory of pseudodifferential operators on the real domain to the
  analytic category. The basic calculus was developed by L.~Boutet de
  Monvel and P.~Kr\'ee \cite{BoKr67}. K.G.~Andersson \cite{An70} and L.~H\"ormander
  \cite{Ho71c} studied propagation of analytic singularities. The work
  \cite{Ho71c} also introduced the analytic wave
  front set of distributions, a corresponding notion in the framework
  of hyperfunctions had previously been introduced by M.~Sato (see
  \cite{SaKaKa71}). The two works \cite{An70}, \cite{Ho71c} use special sequences of
  cutoff functions, remedying for the lack of analytic functions with
  compact support. Such special sequences have an earlier history,
  see L.~Ehrenpreis \cite{Ehr60}, S.~Mandelbrojt
  \cite{Ma42, Ma52}. The book \cite{Tr80a} of
  F.~Treves gives the theory of analytic pseudodifferential operators,
  with the help of such cutoffs.
\item A second approach is based on the representation of
  distributions and more generally hyperfunctions as sums of boundary
  values of holomorphic functions. The main work in this direction is
  the one of M.~Sato, T.~Kawai and M.~Kashiwara \cite{SaKaKa71}.
\item A third approach is to work with Fourier transforms that have
  been modified by the introduction of a Gaussian (avoiding the use of
  the special cutoffs mentioned above. Such transforms come under
  different names: FBI, Bargmann-Segal, Gabor, wavepacket
  .... transforms.  Microlocal properties are now described in terms
  of exponential growth/decay of the transformed functions. In the
  context of analytic microlocal analysis they were introduced and
  used by D.~Iagolnitzer, H.~Stapp \cite{IaSt69}, J.~Bros, Iagolnitzer
  \cite{BrIa75}. This is the method we follow here. See \cite{Sj82, Ma02a}.
\end{itemize}

\medskip
\noindent
The aim of this part of the text is to explain the basic ingredients in the approach of \cite{Sj82}, that was preceded by some work on
propagation of analytic singularities for boundary value problems, see \cite{Sj81}. The main observation is that an FBI-transform
produces holomorphic functions whose exponential growth rate reflect the regularity and that such transforms are Fourier integral operators
with complex phase functions. This leads to a calculus of Fourier integral operators and pseudodifferential operators in the complex
domain via a Egorov theorem. In this calculus oscillatory integrals are systematically replaced by contour integrals, leading to ``Cauchy
integral operators''.

\medskip
\noindent
This part of the text will split roughly into 4 unequal parts:
\begin{itemize}
\item In Sections 2.2--2.5 we discuss pseudodifferential operators and
  Fourier integral operators acting on exponentially weighted spaces of
  holomorphic functions.
\item In Sections 2.6, 2.7 we introduce FBI (generalized Bargmann-)
  transforms and the analytic wave front set of a distribution.
\item The sections 2.8, 2.9 are devoted to some applications: propagation
  of singularities, construction of exponentially accurate quasi-modes
  for non-self-adjoint differential operators.
\item In Section 2.10 we give a very quick review of related developments.
\end{itemize}

\section{Classical analytic symbols and pseudodifferential operators.}
\label{cas}
\setcounter{equation}{0}
Let $\Omega \subset {\bf C}^n$ be open, $\phi \in C(\Omega ;{\bf R})$. By definition, the function $u=u(z;h)$ on $\Omega \times
  ]0,h_0[$ belongs to $H_\phi^{\mathrm{loc}}(\Omega )$ if
\begin{itemize}
\item $u(\cdot ;h)\in \mathrm{Hol}(\Omega )$, for all $h$, where
  $\mathrm{Hol\,}(\Omega )$ denotes the space of holomorphic functions
  on $\Omega $.
\item $\forall K\Subset \Omega $, $\varepsilon >0$, $\exists C>0$ such
  that $|u(z;h)|\le Ce^{(\phi (z)+\varepsilon )/h}$, $z\in K$.
\end{itemize}
When $u\in H_0(\Omega )$, we say that $u$ is an analytic symbol. When
$u={\cal O}(h^{-m})$ locally uniformly on $\Omega $, we say that $u$
is of finite order $m\in {\bf R}$.

\medskip
\noindent
We frequently identify equivalent elements of $H_\phi^\mathrm{loc}(\Omega )$, where the equivalence
$u\sim v$ of $u,v\in H^\mathrm{loc}_\phi (\Omega )$ means that there exists $C^0(\Omega )\ni \phi _0<\phi $, such that
$u-v\in H^\mathrm{loc}_{\phi _0}(\Omega )$. When $\Omega $ is pseudoconvex and the weights are pluri-subharmonic, we can represent equivalence
classes by functions $u\in L^2_\mathrm{loc}(\Omega )$ for which $\|e^{-\phi /h}u\|_{L^2(K)}\le C_{\varepsilon ,K}e^{\varepsilon /h}$,
$\|e^{-\phi_0 /h}\overline{\partial }u\|_{L^2(K)}\le C_{\varepsilon ,K}e^{\varepsilon /h}$
$\forall$ $\varepsilon >0$, $K\Subset \Omega $. Indeed by applying H\"ormander's method of solving the $\overline{\partial }$ equation it
is easy to make such a function $u$ holomorphic by adding a correction which is locally exponentially small compared to $e^{\phi /h}$.

\medskip
\noindent
By $H_{\phi ,x_0}$ we denote the intersection of all spaces $H_\phi (\Omega )$ where $\Omega $ is a small neighborhood of
$x_0\in {\bf C}^n$ and $\phi $ is defined in some fixed neighborhood of $x_0$. We have a corresponding equivalence relation.

\paragraph{Classical analytic symbols} (Boutet de Monvel, Kr\'ee \cite{BoKr67}). We restrict the attention to symbols of order $0$. Let
$a_k\in \mathrm{Hol\,}(\Omega)$, $k=0,1,...$ and assume that for every $\widetilde{\Omega }\Subset \Omega $, $\exists C=C_{\widetilde{\Omega }}>0$ such that
\begin{equation}
\label{cas.1}
|a_k(z)|\le C^{k+1}k^k,\ z\in \widetilde{\Omega }.
\end{equation}
$a=\sum_0^\infty a_k(z)h^k$ is called a classical analytic symbol.

\medskip
\noindent
We have a realization of $a$ on $\widetilde{\Omega }$ by
$$
a_{\widetilde{\Omega }}(z;h)=\sum_{0\le k\le (eC_{\widetilde{\Omega }}h)^{-1}}a_k(z)h^k.
$$
For $0\le k\le (eC_{\widetilde{\Omega }}h)^{-1}$ we have
$$
|a_k(z)|h^k\le C_{\widetilde{\Omega }} (C_{\widetilde{\Omega}}hk)^k\le C_{\widetilde{\Omega }}e^{-k},
$$
so the defining sum above converges geometrically and
$|a_{\widetilde{\Omega }}(z;h)|\le C_{\widetilde{\Omega }}e/(e-1)$.

\medskip
\noindent
If $\widehat{\Omega }\supset \widetilde{\Omega }$ is another
relatively compact subset of $\Omega $, then $a_{\widehat{\Omega }}$
and $a_{\widetilde{\Omega }}$ are equivalent on $\widetilde{\Omega }$.
It is sometimes convenient to consider classical symbols of the form
$$a=\sum_0^\infty a_k(z)h^k,\ a_k\in \mathrm{Hol\,}(\Omega )$$ without
the growth condition (\ref{cas.1}).

\medskip
\noindent
Let
$$p(x,\xi ;h)=\sum_0^\infty h^kp_k(x,\xi ),\ q(x,\xi ;h)=\sum_0^\infty
h^kq_k(x,\xi )$$
be classical symbols defined near $(x_0,\xi _0)\in {\bf C}^{2n}$. Denote
by $p(x,hD;h)$, $q(x,hD;h)$ the corresponding formal pseudodifferential operators. The formal
composition of $p$ and $q$ is defined by
$$
p\# q=\sum_{\alpha \in {\bf N}^n}\frac{{h^{\abs{\alpha}}}}{\alpha !}\partial _\xi
^\alpha p(x,\xi ;h) D_x^\alpha q(x,\xi ;h),
$$
which is a finite sum for each power of $h$. Here, we use standard
PDE-notation, $D_x=i^{-1}\partial _x$,
\[
\begin{split}
\partial _x^\alpha =\partial _{x_1}^{\alpha _1}\cdots\partial
_{x_n}^{\alpha _n},\ |\alpha |=|\alpha |_{\ell^1}=\alpha _1+...+\alpha
_{n},\hbox{ for }\alpha =(\alpha _1,...,\alpha _n)\in {\bf N}^n.
\end{split}
\]

\medskip
\noindent
When $p,q$ are polynomials in $\xi $, the differential operators $p(x,hD;h)$,
$q(x,hD;h)$ are well defined and
$$
p(x,D_x;h)\circ q(x,hD_x;h)=(p\# q)(x,hD;h).
$$
If $r$ is a third symbol, also polynomial in $\xi $, it follows that
\begin{equation}\label{cas.2}
(p\# q)\#r=p\# (q\# r).
\end{equation}
In general, we can approximate $p,q,r$ with finite Taylor polynomials
at any given point and see that we still have (\ref{cas.2}).

\medskip
\noindent
To $p$, we associate
\[\begin{split}&A(x,\xi ,hD_x;h)=p(x,\xi +hD_x;h)=\\ &\sum_{\alpha
}\frac{h^\alpha }{\alpha !}\partial _\xi ^\alpha p(x,\xi )D_x^\alpha
=\sum_{k=0}^\infty h^kA_k(x,\xi ,D_x),\end{split}\]
where
\begin{equation}\label{cas.3}
A_k=\sum_{\nu +|\alpha |=k}\frac{1}{\alpha !}(\partial _\xi ^\alpha
p_\nu )(x,\xi )D_x^\alpha
\end{equation}
is a differential operator of order $\le k$.

\medskip
\noindent
Formally, $A=e^{-ix\cdot \xi /h}\circ p(x,hD_x;h)\circ e^{ix\cdot
\xi /h}$ which is exact and well defined, when $p$ is a polynomial in
$\xi $. Let $B=q(x,\xi +hD_x;h)=\sum_0^\infty h^\ell B_\ell$. Then
$C=A\circ B$ is well defined by $C=\sum_0^\infty h^mC_m$,
$C_m=\sum_{k+\ell =m}A_k\circ B_\ell$. By Taylor approximation with
polynomials in $\xi $, we see that
$$
C=r(x,\xi +hD_x;h),\text{ if }r=p\# q.
$$
\paragraph{Quasi-norms} Let $\Omega _t\Subset {\bf C}^{2n}$, $0\le t\le
t_0$, $t_0>0$ be a family of open neighborhoods of a point $(x_0,\xi
_0)$ such that
$$(y,\xi )\in \Omega _s\hbox{ and }|x-y|_{\ell ^\infty }<t-s \Longrightarrow
(x,\xi )\in \Omega _t,$$
whenever $0\le s\le t\le t_0$. Here,
$$
|x|_{\ell^\infty }=\sup |x_j|,\ x=(x_1,...,x_n)\in {\bf C}^n.
$$
Then $D_x^\alpha $ is a bounded
operator: $B(\Omega _t)\to B(\Omega _s)$ where $B(\Omega )$ denotes
the space of bounded holomorphic functions on $\Omega $. Moreover, by
the Cauchy inequalities,
\[
\| D_x^\alpha \|_{t,s}:=\| D_x^\alpha \|_{{\cal L}(B(\Omega _t),B(\Omega _s))}\le
\frac{\alpha !}{(t-s)^{|\alpha |}}\le \frac{C_0^{|\alpha |}|\alpha
  |^{|\alpha |}}{(t-s)^{|\alpha |}},
\]
for some constant $C_0>0$.

\medskip
\noindent
If $\Omega _{t_0}$ is a relatively compact subset of the domain of
definition of $p$, then on $\Omega _{t_0}$,
$$
|\partial _\xi^\alpha p_\nu  |\le C^{1+\nu +|\alpha |}\nu ^\nu \alpha
! .
$$
Hence, with a new constant
$$
\| \frac{1}{\alpha !}\partial _\xi ^\alpha pD_x^\alpha  \|_{t,s}\le
C^{1+\nu +|\alpha |}\nu ^\nu \frac{|\alpha |^{|\alpha
    |}}{(t-s)^{|\alpha |}}.
$$

\medskip
\noindent
The number of terms in (\ref{cas.3}) is $\le (1+k)^{n+1}$, so with a
new constant $C>0$, we have
\begin{equation}\label{cas.4}
\|A_k\|_{t,s}\le \frac{C^{k+1}k^k}{(t-s)^k},\ 0\le s<t\le t_0.
\end{equation}

\medskip
\noindent
Conversely, if $p$ is a classical symbol such that (\ref{cas.4})
holds for some $C>0$, then $p$ is a classical analytic symbol near $(x_0,\xi _0)$. In fact,
since $p_k=A_k(1)$, we get for some new $C>0$ that
\begin{equation}\label{cas.5}
\sup_{\Omega _{t_0/2}}|p_k|\le C^{k+1}k^k.
\end{equation}

\medskip
\noindent
Put $f(A)=(f_k(A))_{k=0}^\infty $, where $f_k(A)$ is the smallest
constant $\ge 0$ such that
$$
\|A_k\|_{t,s}\le f_k(A)k^k(t-s)^{-k},\ 0\le s<t\le t_0.
$$
When (\ref{cas.4}) holds, $f_k(A)$ is of at most exponential
growth.

\medskip
\noindent
Let $B=\sum_{0}^\infty h^kB_k$ be an operator of the same
type, so that $B_k$ is a differential operator of order $\le k$.
\begin{lemma}\label{cas1}
If $C=A\circ B$, then $f_k(C)\le \sum_{\nu +\mu =k}f_\nu (A)f_\mu (B)$
or in other terms, $f(C)\le f(A)*f(B)$.
\end{lemma}
\begin{proof} We have
$C_k=\sum_{\nu +\mu =k}A_\nu \circ B_\mu $ and for $0\le s<r<t\le
t_0$:
$$
\|A_\nu \circ B_\mu \|_{t,s}\le f_\nu (A)f_\mu (B)\frac{\nu ^\nu \mu ^\mu
}{(r-s)^\nu (t-r)^\mu }.
$$
Choose $r$ such that
$$
r-s=\frac{\nu }{\nu +\mu }(t-s),\ t-r=\frac{\mu }{\nu +\mu }(t-s).
$$
Then
$$\|A_\nu \circ B_\mu \|_{t,s}\le f_\nu (A)f_\mu (B) \frac{(\nu +\mu
  )^{\nu +\mu } }{(t-s)^{\nu +\mu }},$$
$$\|C_k\|_{t,s}\le \left(\sum_{\nu +\mu =k}f_\nu (A)f_\mu (B)
\right)\frac{k^k}{(t-s)^k}.$$\end{proof}

\medskip
\noindent
For $\rho >0$, put
$$\|A\|_\rho =\sum_0^\infty \rho ^kf_k(A).$$
Then (\ref{cas.4}) holds iff $\|A\|_\rho <\infty $ for $\rho >0$ small
enough.
\begin{lemma}\label{cas2}
Let $C=A\circ B$. If $\|A\|_\rho ,\, \|B\|_\rho <\infty $, then
$\|C\|_\rho <\infty $ and we have $\|C\|_\rho \le \|A\|_\rho \|B\|_\rho $.
\end{lemma}
\begin{proof}
By Lemma \ref{cas1}, we have pointwise with respect to $k$:
$$
(\rho ^kf_k(C))_0^\infty \le (\rho ^kf_k(A))_0^\infty * (\rho ^kf_k(B))_0^\infty
$$
and we have the corresponding inequality for the $\ell ^1$-norms.\end{proof}

\medskip
\noindent
If $p(x,\xi ;h)$ is a classical symbol on a neighborhood of
$\overline{\Omega }_{t_0}$, we put $\|p\|_\rho = \|A\|_\rho  $. If $p$
is a classical analytic symbol then there exists $\rho >0$ such that $\|p\|_\rho <\infty $
and similarly for $q$ corresponding to $B$. Since $p\# q$ corresponds
to $A\circ B$, we obtain $\|p\# q\|_\rho \le \|p\|_\rho \|q\|_\rho  $
and we conclude that $p\# q$ is a classical analytic symbol in $\Omega
_{t_0}$. Next we give a semi-classical formulation of a fundamental
result of L.~Boutet de Monvel, P.~Kr\'ee \cite{BoKr67}:
\begin{theo}\label{cas3}
Let $p$ be an elliptic classical analytic symbol ($p_0\ne 0$) on a neighborhood of
$\overline{\Omega }_{t_0}$ and let $q$ be the classical symbol given by
$p\# q=1$. Then $q$ is a classical analytic symbol in $\Omega _{t_0}$.
\end{theo}
\begin{proof}
Let $q_0=1/p_0$, so that $q_0$ is a classical analytic symbol. Then $p\# q_0=1-r$
where $r$ is a classical analytic symbol of order $-1$ in the sense that its asymptotic
expansion starts with a term in $h$. Consequently $\|r\|_\rho <1/2
$ if $\rho >0$ is small enough. We have
$$
q=q_0\# (1+r+r\#r+...),
$$
so that
$$
\|q\|_\rho \le \|q_0\|_\rho (1+\|r\|_\rho + \|r\|_\rho ^2+...)\le
2\|q_0\|_\rho <\infty . $$
\end{proof}

\section{Stationary phase -- steepest descent}
\setcounter{equation}{0}
Let $B=B_{{\bf R}^n}(0,1)$ be the open unit ball in ${\bf R}^n$ and put
$$
\widetilde{B}=\{\lambda x;\, x\in \overline{B},\
\lambda \in \comp,\, \abs{\lambda}\leq 1\}.
$$
\begin{theo}
\label{stp1}
There exist a constant $C>0$ depending only on the dimension, such
that for all $N\in {\bf N}$, $0<h\le 1$, $u\in
\mathrm{Hol\,}(\mathrm{neigh\,}(\widetilde{B}))$,
$$
\int_B e^{-x^2/(2h)}u(x)dx=\sum_{\nu =0}^{N-1}(2\pi
)^{\frac{n}{2}}h^{\frac{n}{2}+\nu }\frac{1}{\nu !}\left(\frac{1}{2}
  \Delta  \right)^\nu u(0)+R_N(h),
$$
where
$$
|R_N(h)|\le Ch^{\frac{n}{2}+N}(N+1)^{\frac{n}{2}}N! 2^N\sup_{\widetilde{B}}|u(z)|.
$$
\end{theo}
We omit the proof and refer to \cite{Sj82}, Chapter 2.

\medskip
\noindent
{\bf Example 3.2}.
\label{stp2}
Consider
$$
J(h)=\left(\frac{h}{2\pi } \right)^n\iint_{|x|\le C_1\atop \xi
  =-C_2i\overline{x}}e^{-ix\cdot \xi /h}u(x,\xi )dxd\xi .
$$
Then,
\[
\begin{split}
  J(h)&=\sum_0^{N-1}\frac{1}{k!}\left(\frac{h}{i}\sum_1^n
    \frac{\partial }{\partial x_j}\frac{\partial }{\partial \xi _j}
  \right)^ku(0,0)+R_N(h)\\
&=\sum_{|\alpha |\le N-1}\frac{1}{\alpha !}\left(\frac{h}{i}
\right)^{|\alpha |} \left(\partial _x^\alpha \partial _\xi ^\alpha u \right)(0,0)+R_N(h),
\end{split}
\]
$$
|R_N(h)|\le C(n)(N+1)^nN! \left(\frac{h}{C_1^2C_2}
\right)^N\sup_{|x|\le C_1\atop |\xi |\le C_1C_2}|u(x,\xi )|.
$$
This follows from Theorem \ref{stp1}, some calculations and the following
three observations:
\begin{itemize}
\item $\Gamma :$ $\xi =(C_2/i)\overline{x}$ is a maximally totally
  real subspace of ${\bf C}^{2n}$, hence $\simeq {\bf R}^{2n}$ after a
  complex linear change of coordinates.
\item The restriction of $e^{-ix\cdot \xi /h}$ to $\Gamma $ is equal
  to $e^{-C_2|x|^2/h}$.
\item The corresponding restriction of $i^{-1}\partial
  _x\cdot \partial _\xi $ is equal to
$$
\frac{1}{i}\partial _x\cdot \frac{i}{C_2}\partial
_{\overline{x}}=\frac{1}{4C_2}\Delta _{{\rm Re}\,x,{\rm Im}\,x}.
$$
\end{itemize}

\paragraph{Non-quadratic case.} The holomorphic version of the Morse
lemma is the following:
\begin{lemma}\label{stp3}
Let $\phi \in \mathrm{Hol\,}(\mathrm{neigh\,}(z_0,{\bf C}^n))$, $\phi
'(z_0)=0$, $\det \phi ''(z_0)\ne 0$. Then there exist local
holomorphic coordinates $\widetilde{z}_1,...,\widetilde{z}_n$, centered at $z_0$ such that
$$
\phi (z)=\phi (z_0)+\frac{1}{2}(\widetilde{z}_1^2+...+\widetilde{z}_n^2).
$$
\end{lemma}
The main ingredient in the standard proof of the Morse lemma in the
real smooth category is the implicit function theorem in the same
category. To get the proof of the holomorphic Morse lemma it suffices
to use the holomorphic implicit function theorem.
\begin{theo}
\label{stp4}
Let $0\in V\Subset U\subset {\bf C}^n$, $V,\,U$ open, $\phi \in
\mathrm{Hol\,}(U)$, $\phi (0)=0$, $\phi '(0)=0$, $\phi ''(0)$
non degenerate. Assume that ${\rm Re}\,\phi \ge 0$ on $V_{\bf R}:=V\cap {\bf
  R}^n$, ${\rm \Re}\,\phi >0$ on $\partial V_{\bf R}$, $\phi '(x)\ne 0$ on
$V_{\bf R}\setminus \{0 \}$. Then, for every $C>0$ large enough, there exists
a constant $\varepsilon >0$ such that for every $u\in \mathrm{Hol\,}(U)$,
$$
\int_{V_{\bf R}}e^{-\phi (x)/h}u(x)dx=\sum_{0\le k\le 1/(Ch)}
(2\pi h)^{\frac{n}{2}}\frac{h^k}{k!}\left(\frac{1}{2}\widetilde{\Delta
  }\right)^k\left(\frac{u}{J} \right)(0)+R(\lambda ),
$$
where
$$
|R(h)|\le \frac{1}{\varepsilon }e^{-\frac{\varepsilon }{h}}\sup_U |u(z)|,\
0<h\le 1.
$$
Here, $\widetilde{\Delta }$ denotes the Laplacian in the Morse
coordinates, $J=\det \frac{d\widetilde{z}}{dz}$, $J(0)=(\det \phi
''(0))^{\frac{1}{2}}$, with the choice of the branch of the square
root that tends to 1, when we deform $\phi ''(0)$ to $1$ in the space
of invertible symmetric matrices with real part $\ge 0$.
\end{theo}
\begin{proof}
Up to an exponentially small modification, we may replace the integral
by
\[\begin{split}
&I_\chi =\int_{{\bf R}^n}e^{-\phi (x)/h}u(x)\chi (x)dx,\ \chi \in
C_0^\infty (V_{\bf R}),\\ &\mathrm{supp\,}(1-\chi )\subset \hbox{ small
  neighborhood of }\partial V_{\bf R}.
\end{split}\]
Make a first contour deformation
$\Gamma _\delta :\, V_{\bf R}\ni x\mapsto x+\delta \overline{\phi
  '}(x)$, $0\le \delta \le \delta _0\ll 1$. Along $\Gamma _\delta $ we
have
$$
\phi (z)=\phi (x)+\delta |\phi '(x)|^2+{\cal O}(\delta ^2|\phi
'(x)|^2)\ge \frac{\delta }{C}|z|^2,
$$
when $\delta _0$ is small enough.

\medskip
\noindent
Let $G$ be the $(n+1)$-dimensional contour formed by the union of the $\Gamma _\delta $ for $0\le \delta
\le \delta _0$. Then Stokes' formula gives (with $\chi $ denoting also a suitable smooth extension to the complex domain),
$$
I_\chi =\int_{\Gamma _{\delta _0}}e^{-\phi (z)/h}u(z)\chi (z)dz-\int_G
d(e^{-\phi /h}u(z)\chi (z)dz).
$$
The last integral is equal to
$$\int_{G\cap \mathrm{neigh\,}(\partial V_{\bf R})} e^{-\phi (z)/h}u(z){\overline{\partial }\chi
  (z)\wedge dz}.$$
When estimating the integral over $\Gamma _{\delta _0}$, we can
restrict the attention to a small neighborhood of $0$ and then use
Morse coordinates for which $\phi =\frac{1}{2}\widetilde{z}^2$. Since
$\Re \phi \asymp |\widetilde{z}|^2$ along $\Gamma _{\delta _0}$, we
see that $\Gamma _{\delta _0}$ must be of the form
$\widetilde{y}=k(\widetilde{x})$
($\widetilde{z}=\widetilde{x}+i\widetilde{y}$), where $|k'|\le \theta
<1$, $k(0)=0$. (Use the implicit function theorem, to see that the
projection $\Gamma _{\delta _0}\ni \widetilde{z}\mapsto \widetilde{x}$
is a diffeomorphism near $0$.) The last step is then to deform the
contour $\widetilde{y}=k(\widetilde{x})$ to $\widetilde{y}=0$ in the
simplest possible way and to apply Theorem \ref{stp1}.
\end{proof}

\section{Contour integrals and Fourier transforms}
\setcounter{equation}{0}
\paragraph{a. Remarks about real quadratic forms on ${\bf C}^n$.} Let $q$ be a real quadratic form on ${\bf C}^n\simeq {\bf
  R}^{2n}$. Let $\mathrm{sign\,}(q)=(m_+,m_-)$ where $m_\pm =m_\pm (q)$ are given by
$$
q=\sum_1^{m_+}\xi _j^2-\sum_{m_++1}^{m_++m_-}\xi _j^2,
$$
for suitable real-linear coordinates on ${\bf C}^n$. We know
that $m_+$ ($m_-$) is the largest possible dimension of a real-linear subspace
on which $q$ is positive (negative) definite.

\medskip
\noindent
Using the complex structure, put $Jq(x)=q(ix)$, so that $J^2q=q$
(since $q$ is even).

\medskip
\noindent
Notice that $q$ is pluriharmonic iff $Jq=-q$.

\medskip
\noindent
We say that $q$ is Levi if $Jq=q$.

\medskip
\noindent
In general we have the decomposition $$q=h+\ell =2\Re (\sum a_{j,k}z_jz_k)+\sum b_{j,k}\overline{z}_jz_k,$$ where
$h=(1-J)q/2$ is pluri-harmonic and $\ell =(1+J)q/2$ is Levi.
\begin{prop}\label{cif1}
Let $q$ be a pluri-subharmonic quadratic form on ${\bf C}^n$. Then
\begin{itemize}
\item[(a)] $m_+(q)\ge m_-(q)$
\item[(b)] If $q$ is non-degenerate of signature $(n,n)$, then the same
  fact holds for every pluri-subharmonic quadratic form $\widetilde{q}\le q$.
\end{itemize}
\end{prop}
\begin{proof} The pluri-subharmonicity of $q$ means that $\ell \ge 0$.
\medskip
\noindent
(a) Let $L\subset {\bf C}^n$ be a real-linear subspace of
dimension $m_-=m_-(q)$ such that ${{q}_\vert}_{L}<0$. Use the
decomposition $q=h+\ell$. Then $h(x)=q(x)-\ell (x)<0$ for $0\ne x\in
L$. Consequently, $h(ix)>0$, so $q(ix)=h(ix)+\ell (ix)>0$. Thus $q$ is
positive definite on the $m_-$-dimensional space $iL$, so $m_+\ge m_-$.
\medskip
\noindent
(b) Now assume that $m_+=m_-=n$. Let $\widetilde{q}\le q$ be pluri-subharmonic
and choose the subspace $L$ as in (a). Then $\widetilde{q}$ is
negative definite on $L$ so $m_-(\widetilde{q})\ge m_-(q)=n$ and from
the part (a) of the proposition we conclude that $\widetilde{q}$ has
signature $(n,n)$.\end{proof}
\paragraph{b. Fundamental lemma.}
\begin{lemma}\label{cif2}
Let $\phi \in C^\infty (\mathrm{neigh\,}((0,0),{\bf C}^{n+k});{\bf
  R})$ be pluri-subharmonic. Assume that $\nabla _y\phi (0,0)=0$ and that
$\nabla _y^2\phi(0,0)$ is non\-dege\-ne\-rate of signature $(k,k)$. For $x\in
\mathrm{neigh\,}(0,{\bf C}^n)$, let $y(x)\in \mathrm{neigh\,}(0,{\bf
  C}^k)$ be the unique critical point of $\phi (x,\cdot )$, so that
$y(x)$ is a smooth function of $x$ by the implicit function
theorem. Then the critical value of $y\mapsto \phi (x,y)$,
$$\Phi (x)=\phi (x,y(x))=\mathrm{vc}_y\phi (x,y)$$ is
pluri-subharmonic. If $\widetilde{\phi }\le \phi $ is pluri-subharmonic with $\widetilde{\phi
}(0,0)=\phi (0,0)$, then $\nabla _y^2\widetilde{\phi }(0,0)$ is also
non-degenerate of signature $(k,k)$ and
$$
\mathrm{vc}_{y}\widetilde{\phi } (x,y)\le \mathrm{vc}_{y}\phi  (x,y),
\hbox{ for }x\in \mathrm{neigh\,}(0,{\bf C}^n).
$$
\end{lemma}
\begin{proof}
Let $L\subset {\bf C}^k$ be a subspace of real dimension $k$ such that
${{\nabla _y^2\phi (0,0)}_\vert}_{L}<0$. Then ${{\nabla _y^2\phi
    (0,0)}_\vert}_{iL}>0$. For $t\in \mathrm{neigh\,}(0,iL)$, put
$L_t=t+L$, so that the $\Gamma _t$ form a foliation of a neighborhood
of $0\in {\bf C}^k$. Then, it is well known that
$$
\phi (x,y(x))=\inf_t\sup_{y\in \Gamma _t}\phi (x,y),\ x\in
\mathrm{neigh\,}(0,{\bf C}^n).
$$
If $\widetilde{\phi }\le \phi $ is as in the statement of the lemma,
we have ${{\nabla _y^2\widetilde{\phi }(0,0)}_\vert}_{L}<0$, so
$\nabla _y^2\widetilde{\phi }(0,0)$ is non-degenerate of signature
$0$. Then $y\mapsto \widetilde{\phi }(x,y)$ has a non-degenerate
critical point $\widetilde{y}(x)$ and we have the same mini-max
formula as for $\phi $:
$$
\widetilde{\phi } (x,y(x))=\inf_t\sup_{y\in \Gamma _t}\widetilde{\phi } (x,y),\ x\in
\mathrm{neigh\,}(0,{\bf C}^n).
$$
It is then clear that $\widetilde{\phi }(x,\widetilde{y}(x))\le \phi
(x,y(x)).$
\medskip
\noindent
Replacing $\phi $, $\widetilde{\phi }$ by their quadratic Taylor polynomial $\phi
^{(2)}(x,y)$, $\widetilde{\phi }^{(2)}(x,y)$ at $(0,0)$, and the critical points by their linear
Taylor polynomials $y^{(1)}(x)$ and $\widetilde{y}^{(1)}(x)$, we see
that $\phi ^{(2)}(x,y^{(1)}(x))$, $\widetilde{\phi }
^{(2)}(x,\widetilde{y}^{(1)}(x))$ are the quadratic Taylor polynomials
of $\phi (x,y(x))$, $\widetilde{\phi }(x,\widetilde{y}(x))$. Taking
$\widetilde{\phi }^{(2)}$ pluri-harmonic it is clear that $\widetilde{\phi
}^{(2)}(x,\widetilde{y}^{(1)}(x))$ is pluri-harmonic and $\le \phi ^{(2)}(x,y
^{(1)}(x))$, so the latter is pluri-subharmonic. This shows that $\mathrm{vc}_y\phi
(x,y)$ has a positive semi-definite Levi form at $0$. The same argument
now works with $0$ replaced by any other point in
$\mathrm{neigh\,}(0,{\bf C}^n)$ and we get the desired plurisubharmonicity.
\end{proof}
\paragraph{c. Contour integration.} Let $\phi (y)\in C^\infty (\mathrm{neigh\,}(0,{\bf C}^k);{\bf R})$. Assume that $0$ is
a ``col'' for $\phi $ in the sense that $\nabla _y\phi (0)=0$ and $\nabla _y^2\phi (0)$ is non-degenerate of signature $(k,k)$. Consider
a smooth contour $\Gamma :\mathrm{neigh\,}(0,{\bf R}^k)\to \mathrm{neigh\,}(0,{\bf C}^k)$ with $\Gamma (0)=0$, $d\Gamma $
injective. We say that $\Gamma $ is a good contour if
$$
\phi (y)-\phi (0)\le -\frac{1}{C}|y|^2,\ y\in \Gamma .
$$

\medskip
\noindent
If $u\in H_{\phi ,0} $ i.e. an element of $H_\phi(\mathrm{neigh\,}(0,{\bf C}^k))$, then
$$
I_\Gamma (h)=e^{-\phi (0)/h}\int_{\Gamma }u(y;h)dy
$$
is well-defined up to an exponentially small ambiguity (and also up to
a factor $\pm$ depending on a choice of orientation, that we shall
simply forget). As we have seen, a second good contour passing through
$0$ can be deformed to $\Gamma $ within the set of such good contours.

\medskip
\noindent
Now take $\phi (x,y)\in C^\infty (\mathrm{neigh\,}((0,0),{\bf C}^{n+k});{\bf R})$ with $\phi (0,y)$ as above. If $\Gamma $ is a
good contour for the latter function and $u\in H_{\phi, (0,0) }$, then
by deforming $\Gamma $ into an $x$-dependent good contour for $\phi(x,\cdot )$, we see that
$$
U(x;h)=\int _\Gamma u(x,y;h)dy
$$
is a well defined element of $H_{\Phi ,0}$, where $\Phi(x)=\mathrm{vc}_y\phi (x,y)$.

\medskip
\noindent
When working with differential forms of other degrees, we may be
interested in other signatures than $(k,k)$. Also, for instance when
composing Fourier integral operators, one is frequently in the situation of integrating
along a good contour with respect to one group of variables and then
for the resulting integral we want a good contour for the last group
of variables. The following discussion (that we state only for quadratic
forms) shows that this will always work as well as one can possibly
hope for.

\medskip
\noindent
This has nothing to do with the complex structure, so we consider a
decomposition $x=(x',x'')\in {\bf R}^n$, $x'\in {\bf R}^{n-d}$,
$x''\in {\bf R}^d$. Let $q$ be a quadratic form on ${\bf R}^n$ such
that $q''(x''):=q(0,x'')$ is a non-degenerate quadratic form on ${\bf R}^d$. Then
$
x''\mapsto q(x',x'')
$
has a unique critical point $x''=x''(x')$ depending linearly on
$x'$. Consequently, the corresponding critical value
$q'(x')=q(x',x''(x'))$ is a quadratic form on ${\bf R}^{n-d}$. Let
$(m_+(q),m_-(q))$ be the signature of $q$ and denote the signatures of
$q'$ and $q''$ similarly. Then by assumption, $m_+(q'')+m_-(q'')=d$.
\begin{prop}\label{cif3}
Under the above assumptions we have
\begin{equation}\label{cif.1}
m_+(q)=m_+(q')+m_+(q''),\ m_-(q)=m_-(q')+m_-(q'').
\end{equation}
If $L_-'$, $L_-''$ are subspaces of ${\bf R}^n$ of dimension $m_-(q')$
and $m_-(q'')$ respectively such
that
${{q'}_\vert}_{L_-'}$,
${{q''}_\vert}_{L_-''}$ are negative definite, and we put
$L_-=\{ (x',x''(x')+x'';\, x'\in L_-',\, x''\in L_-'' \}$, then
${{q}_\vert}_{L_-}$ is negative definite.
\end{prop}

\begin{proof} After
the change of variables $x'=\widetilde{x}'$,
$x''=x''(\widetilde{x}')+\widetilde{x}''$, we are reduced to the case
when $x''(x')\equiv 0$. This means (after dropping the tildes on the
new variables) that
$$q(x)=q'(x')+q''(x'')$$
and the conclusion follows.
\end{proof}

\paragraph{d. Application to Fourier transforms.} Let $\phi \in C^\infty (\mathrm{neigh\,}(x_0,{\bf C}^n);{\bf R})$ be pluri-subharmonic with
$\phi''(x_0)$ non-degenerate of signature $(n,n)$. Let $\xi_0=\frac{2}{i}\frac{\partial \phi }{\partial x}(x_0)$.
For $\xi \in \mathrm{neigh\,}(\xi _0,{\bf C}^n)$, we put
$$
\phi ^*(\xi )=\mathrm{vc}_x (\phi (x)+\Im (x\cdot \xi )),
$$
where the critical point $x=x(\xi )$ is given by
$$
\xi =\frac{2}{i}\partial _x\phi (x),\ x(\xi _0)=x_0.
$$
Guided by the Fourier inversion formula (that we shall study below), we look at
$$
(y,\xi )\mapsto -\Im (x\cdot \xi )+\Im (y\cdot \xi )+\phi (y)
$$
which is pluri-subharmonic with the critical point $y=x$,
$\xi =\frac{2}{i}\partial_x\phi (x)$ and the corresponding critical value $\phi (x)$. The
critical point is non-degenerate of signature $(2n,2n)$ since we have the good contour
$$
\Gamma _R(x):\ \xi =\frac{2}{i}\partial _x\phi (x)+iR\overline{(x-y)},\ |x-y|<r,
$$
parametrized by $y\in B_{{\bf C}^n}(x,r)$. Indeed by Taylor expanding, we get:
$$
-\Im ((x-y)\cdot \xi )+\phi (y)=\phi (x)-(R-{\cal O}(1))|x-y|^2,\
(y,\xi )\in \Gamma _R(x).
$$
with the ``${\cal O}(1)$'' uniform in $R$. Hence $\Gamma _R$ is a good contour for $R$ large enough and $r>0$ small enough.

\medskip
\noindent
Applying Proposition \ref{cif3}, we now see that
$$
\xi \mapsto -\Im (x\cdot \xi )+\phi ^*(\xi )
$$
has a non-degenerate critical point $\xi =\xi (x)$ of signature $(n,n)$ at $\xi (x)=\frac{2}{i}\partial _x\phi (x)$ and
$$
\phi (x)=\mathrm{vc}_\xi (-\Im (x\cdot \xi )+\phi ^*(\xi )).
$$
This is a standard inversion formula for Legendre transforms when viewing $\phi ^*$ as the Legendre transform of $\phi $.

\medskip
\noindent
Using a good contour, we can define the Fourier transform
$$
{\cal F}u(\xi ;h)=\int_{\Gamma _\xi  }\underbrace{e^{-ix\cdot \xi /h
  }u(x;h)}_{\in H_{\phi (\cdot )+{\rm Im}\,((\cdot )\cdot \xi ) }}dx\in H_{\phi^*,\xi _0^*}.
$$

\medskip
\noindent
For $v\in H_{\phi ^*,\xi _0^*}$, we put
$$
{\cal G}v(x;h)=\frac{1}{(2\pi h)^n}\int_{\Gamma _x^*}e^{ix\cdot \xi /h
}v(\xi )d\xi ,
$$
where $\Gamma _x^*$ is a good contour such that
$$
\phi ^*(\xi )-\Im (x\cdot \xi )-\phi (x)\le -\frac{1}{C}|\xi -\xi
(x)|^2,\ \xi (x)=\frac{2}{i}\partial _x\phi (x).
$$
\begin{prop}\label{cif4}
For $u\in H_{\phi ,x_0}$, we have $u={\cal G}{\cal F}u$ in $H_{\phi
  _0,x_0}$ (up to equivalence).
\end{prop}
\begin{proof} We have
$$
{\cal G}{\cal F}u(x)=\frac{1}{(2\pi h)^n}\int_{\Gamma
  _x^*}\int_{\Gamma _\xi }e^{i(x-y)\cdot \xi /h}u(y)dyd\xi \hbox{
  (iterated integral).}
$$
Along the composed contour we have (cf Proposition \ref{cif3})
\[\begin{split}
-\Im (x\cdot \xi )+\phi ^*(\xi )&\le \phi (x)-\frac{1}{C}|\xi -\xi
(x)|^2,\ \xi \in \Gamma _x^*,\\
\Im (y\cdot \xi )+\phi (y )&\le \phi^* (\xi )-\frac{1}{C}|y -x(\xi
)|^2,\
y \in \Gamma _\xi ,
\end{split}\]
so
$$
-\Im ((x-y)\cdot \xi )+\phi (y)\le \phi (x)-\frac{1}{C}(|\xi -\xi
(x)|^2+|y-x(\xi )|^2).
$$
\medskip
\noindent
The composed contour is a good contour like $\Gamma _R$.

\medskip
\noindent
Thus, up an exponentially small error, we can replace the
composed contour by $\Gamma _R$ for $R$ large enough and get
\[
\begin{split}
&\frac{1}{(2\pi h)^n}\iint_{\Gamma _R(x)}e^{i(x-y)\cdot \xi
  /h}u(y)dyd\xi =\\
&\left(\frac{R}{i2\pi h}
\right)^n\iint_{|x-y|<r}e^{\frac{2}{h}(x-y)\cdot \partial
  _x\phi (x)-\frac{R}{h}|x-y|^2}u(y)dy\wedge d\overline{y}\\
&=(1+{\cal O}(e^{-Rr^2/h}) )u(x)
\end{split}
\]
by the spherical mean-value property for holomorphic functions.\end{proof}
\section{Pseudodifferential operators and Fourier integral operators}
\label{pf}
\setcounter{equation}{0}
Let $a(x,y,\theta ;h)$ be an analytic symbol defined near
$(x_0,x_0,\xi _0)\in {\bf C}^{3n}$, so that $a\in H_{0,(x_0,x_0,\xi_0)}$. Let $\phi \in C^\infty (\mathrm{neigh\,}(x_0,{\bf C}^n);{\bf
  R})$ with $(2/i)\partial _x\phi (x_0)=\xi _0$. For $u\in H_{\phi
  ,x_0}$, we define $Au\in H_{\phi ,x_0}$ by
$$
Au(x;h)=\frac{1}{(2\pi h)^n}\iint_{\Gamma (x)}e^{i(x-y)\cdot \theta
  /h}a(x,y,\theta ;h)u(y;h)dyd\theta ,
$$
where $\Gamma (x)=\Gamma_R(x)$ is the good contour introduced at the
end of the preceding section so that (for $R$ large enough)
$$
e^{-\phi (x)/h}\left| e^{i(x-y)\cdot \theta /h} \right| e^{\phi
  (y)/h}\le e^{-\frac{1}{h}(R-{\cal O}(1))|x-y|^2}
$$
along $\Gamma (x)$. It follows that
$$Au(x;h)=A_\Gamma u(x;h)=\int k_\Gamma (x,y;h)u(y)L(dy),$$
where
$$
|k_\Gamma (x,y;h)|e^{(-\phi (x)+\phi (y))/h}\le C_\Gamma
h^{-n}e^{-\frac{1}{h}(R-{\cal O}(1))|x-y|^2}.
$$
$A_\Gamma $ is uniformly bounded $L^2_{\phi ,x_0}\to L^2_{\phi,x_0}$. Here, we assume for simplicity that $|a(x,y,\theta ;h)|\le
{\cal O}(1)$. Without that assumption we would need to insert a factor $C_\epsilon e^{\epsilon /h}$ to the right in the last estimate and the
boundedness statement about $A_\Gamma $ has to be modified accordingly.

\medskip
\noindent
We define the symbol of $A$ by
$$
\sigma _A(x,\xi ;h)=e^{-ix\cdot \xi /h}A(e^{i(\cdot )\cdot \xi /h}),\
(x,\xi )\in \mathrm{neigh\,}((x_0,\xi _0),{\bf C}^{2n}).
$$
The method of stationary phase gives
$$\sigma _A(x,\xi ;h)\equiv \sum_{|\alpha |\le 1/(Ch)}\frac{1}{\alpha!}(\partial _\xi ^\alpha D_x^\alpha a)(x,x,\xi ;h)$$
and this is (a realization of) a classical analytic symbol when $a$ is a classical analytic
symbol. Clearly $\sigma _A\equiv a$ when $a$ does not depend on $y$.
\begin{lemma}
\label{pf1}
Assume that $\sigma _A=0$ in $H_{0,(x_0,\xi _0)}$. Then
$\exists b\in H_{0,(x_0,x_0,\xi _0)}$ with values in the $(n-1)$-forms
in $\theta $ such that
$$
e^{i(x-y)\cdot \theta /h}a(x,y,\theta )d\theta \equiv ihd_\theta\left(e^{{i}(x-y)\cdot \theta/h}b\right),
\hbox{ in }H_{-{\rm Im}\,((x-y)\cdot \theta ),(x_0,x_0,\xi _0)}.
$$
Applying the Stokes formula along the good contour, it then follows
that $A=0$ as an operator in $H_{\phi ,x_0}$.
\end{lemma}
\begin{proof} By a simple change of variables,
\[
\begin{split}
(2\pi h)^n\sigma _A(x,\eta )=\iint e^{-iy\cdot \eta
  /h}\underbrace{a(x,x-y,\theta ;h)e^{iy\cdot \theta
    /h}}_{u(x,y,\theta ;h)}dyd\theta \\
={\cal F}_{(y,\theta )\to (\eta ,\theta ^*)}(u)(\eta ,0;h)=v(x,\eta ,0;h),
\end{split}
\]
where $x$ is treated as a parameter and $v:={\cal F}_{(y,\theta )\to
  (\eta ,\theta ^*)}(u)$.

\medskip
\noindent
We have $u\in H_\phi $, $v\in H_{\phi ^*}$,
$\phi =-\Im (y\cdot \theta ) $, $\phi ^*=\Im (\eta \cdot \theta ^*)$ and we
observe that $\phi $ and $\phi ^*$ are pluri-harmonic. Now $v(x,\eta ,0;h)=0$ and
Taylor's formula gives
$$
v(x,\eta ,\theta ^*;h)=\sum_1^n \widehat{v}_j(x,\eta ,\theta
^*;h)\theta _j^*,\ \widehat{v}_j\in H_{\phi ^*},
$$
and $\widehat{v}_j$ depend holomorphically on $x$. By Fourier
inversion
$$
u(x,y,\theta ;h)=\sum_1^n hD_{\theta _j}v_j \hbox{ in }H_\phi ,\
v_j\in H_\phi ,
$$
so $v_j=b_j(x,y,\theta ;h)e^{iy\cdot \theta /h}$, $b_j\in H_0$. Going
back to the original variables, we get the identity in the
lemma.\end{proof}

\paragraph{General remarks about Fourier integral operators.} Let
$$
\phi (z,y,\theta)\in C^2(\mathrm{neigh\,}((z_0,y_0,\theta _0),{\bf C}^{n_z+n_y+n_\theta });{\bf R}),\quad f\in C^2(\mathrm{neigh\,}(y_0,{\bf C}^{n_y});{\bf R})
$$
be pluri-subharmonic and assume that $(y,\theta )\mapsto \phi (z,y,\theta )+f(y)$ has a col at
$(y_0,\theta _0)$. If $a\in H_{\phi ,(z_0,y_0,\theta _0)}$, we can
define $A:H_{f,y_0}\to H_{g,z_0}$ by
$$
Au(z;h)=\int_{\Gamma _1(z)}a(z,y,\theta ;h)u(y)dyd\theta ,
$$
where $g(z)=\mathrm{vc}_{y,\theta }(\phi (z,y,\theta )+f(y)$ and
$\Gamma _1(z)$ is a good contour.

\medskip
\noindent
Let $b(x,z,w;h)\in H_{\psi ,(x_0,z_0,w_0)}$, $x\in {\bf C}^{n_x}$ and
assume that $\psi ,g$ fulfill the same assumptions as $\phi ,f$. Then
for $v\in H_{g,z_0}$, we define $Bv\in H_{k,x_0}$ by
$$
Bv(x;h)=\int_{\Gamma _2(x)}b(x,z,w ;h)v(z)dzdw ,
$$
where $\Gamma _2(x)$ and $k(x)$ denote a good contour and the
critical value respectively, for $(z,w)\mapsto \psi (x,z,w)+g(z)$.

\medskip
\noindent
We can then define $B\circ A:H_{f,y_0}\to H_{k,x_0}$ by
$$
B\circ Au(x;h)=\iiiint_{\Gamma (x)}b(x,z,w)a(z,y,\theta )u(y)dyd\theta dzdw,
$$
where $\Gamma (x)$ is the composed contour given by $(z,w)\in \Gamma
_2(x)$, $(y,\theta )\in \Gamma _1(z)$. It is a good contour for
$$
(z,w,y,\theta )\mapsto \psi (x,z,w)+\phi (z,y,\theta )+f(y).
$$

\medskip
\noindent
Now assume that
\begin{equation}\label{pf.1}
(z,w)\mapsto \psi (x_0,z,w)+\phi (z,y_0,\theta _0)
\end{equation}
has a col at $(z_0,w_0)$. Let $F(x,y,\theta )$ be the critical value
when $(z,y,\theta )$ varies near $(x_0,y_0,\theta _0)$. Then $F$ is
pluri-subharmonic, and knowing that $(z,w,y,\theta )\mapsto \psi +\phi +f$ has col,
we see that
\begin{equation}\label{pf.2}
(y,\theta )\mapsto F(x,y,\theta )+f(y)
\end{equation}
has a col.
Hence, if $\Gamma _3(x,y,\theta )$ is a good contour for
(\ref{pf.1}) and $\Gamma _4(x)$ a good contour for (\ref{pf.2}), the
composed contour
$$
\widetilde{\Gamma }(x):\ (y,\theta )\in \Gamma _4(x),\ (z,w)\in \Gamma
_3(x,y,\theta )
$$
is good for
$$
(z,w,y,\theta )\mapsto \psi (x,z,w)+\phi (z,y,\theta )+f(y).
$$

\medskip
\noindent
By Stokes, we can replace $\Gamma (x)$ in the formula for $B\circ
Au(x)$ by $\widetilde{\Gamma }(x)$ and write
\[\begin{split}
B\circ Au(x;h)&=\iiiint_{\widetilde{\Gamma } (x)}b(x,z,w)a(z,y,\theta
)u(y)dyd\theta dzdw\\
&=\iint_{\Gamma _4(x)}
\underbrace{\left(\iint_{\Gamma _3(x,y,\theta
      )}b(x,z,w)a(z,y,\theta )dzdw \right)}_{=:c(x,y,\theta )\in
  H_{F,(x_0,y_0,\theta _0)}} u(y)dyd\theta
\end{split}
\]
This remark can be applied to the case when $A$, $B$ are pseudodifferential operators and when
combining it with the stationary phase, we get
\begin{theo}\label{pf2}
Let $A,B:H_{\phi ,x_0}\to H_{\phi ,x_0}$ be two pseudodifferential operators. Then $B\circ A$
is a pseudodifferential operator with symbol
$$
\sigma _{B\circ A}(x,\xi ;h)=\sum_{|\alpha |\le
  \frac{1}{Ch}}\frac{1}{\alpha !}h^{|\alpha |}\partial _\xi ^\alpha
\sigma _B(x,\xi ;h)D_x^\alpha \sigma _A(x,\xi ;h).
$$
\end{theo}

\section{FBI-transforms and analytic wavefront sets}
\label{fbi}
\setcounter{equation}{0}
Let $\phi \in \mathrm{Hol\,}(\mathrm{neigh\,}((x_0,y_0),{\bf C}^{2n}))$, $y_0\in {\bf R}^n$ and assume that
\begin{equation}\label{fbi.1}
\begin{split}
\phi '_y(x_0,y_0)&=-\eta _0\in {\bf R}^n,\ \Im \phi
''_{yy}(x_0,y_0)>0,\\ &\det \phi ''_{xy}(x_0,y_0)\ne 0.
\end{split}
\end{equation}
Let $a(x,y;h)$ be an elliptic classical analytic symbol defined near
$(x_0,y_0)$ and let $\chi \in C_0^\infty (\mathrm{neigh\,}(y_0,{\bf R}^n))$ be equal to one near $y_0$. If $u\in {\cal D}'({\bf R}^n)$
(or just defined in a neighborhood of the support of $\chi $), we put
\begin{equation}
\label{fbi.1.5}
Tu(x;h)=\int e^{i\phi (x,y)/h}a(x,y;h) \chi (y)u(y)dy,\ x\in
\mathrm{neigh\,}(x_0,{\bf C}^n).
\end{equation}
\begin{prop}\label{fbi1}
$Tu\in H_\Phi (\mathrm{neigh\,}(x_0))$, where
$$
\Phi =\sup_{y\in \mathrm{neigh\,}(y_0,{\bf R}^n)}-\Im \phi
(x,y)\in C^\infty (\mathrm{neigh\,}(x_0,{\bf C}^n);{\bf R}).
$$
\end{prop}
This is evident since ${\bf R}^n\ni y\mapsto -\Im \phi (x,y)$ has a
non-degenerate maximum at $y=y(x)\in \mathrm{neigh\,}(y_0, {\bf R}^n)$.

\medskip
\noindent
Introduce
$$
\Lambda _\Phi =\{ (x,\frac{2}{i}\partial _x\Phi (x));\; x\in
\mathrm{neigh\,}(x_0,{\bf C}^n) \}
$$
Then (and here we only use that $\Phi $ is real and smooth), the
restriction to $\Lambda _\Phi $ of the complex symplectic 2-form
$\sigma =\sum d\xi _j\wedge dx_j$ is real, so $\Lambda _\Phi $ is an
I-Lagrangian manifold, i.e. a Lagrangian manifold for the real
symplectic form $\Im \sigma $.
\begin{prop}\label{fbi2}
$\Lambda _\Phi =\kappa _T({\bf R}^{2n})$, where
$$\kappa _T:\,
\mathrm{neigh\,}((y_0,\eta _0))\ni (y,-\phi
'_y(x,y))\mapsto (x,\phi '_x(x,y)) \in \mathrm{neigh\,}((x_0,\xi
_0))
$$
is the complex canonical transformation associated to $T$, when
viewed as a Fourier integral operator. Here $(x_0,\xi _0)=\kappa _T(y_0,\eta
_0)=(x_0,(2/i)\partial _x\Phi (x_0))$. In particular ${{\sigma
  }_\vert}_{\Lambda _\Phi }$ is real and non-degenerate. ($\Lambda
_\Phi $ is I-Lagrangian and R-symplectic.) Further, $\Phi $ is strictly pluri-subharmonic.
\end{prop}
\begin{proof} The real critical point of $-\Im \phi
(x,\cdot )$ is characterized by the property that $\eta (x):=-\phi
'_y(x,y(x))$ is real. Further,
$$
\frac{2}{i}\partial _x\Phi (x)=\frac{2}{i}(\partial _x(-\Im \phi
))(x,y(x))=\phi '_x(x,y(x)).
$$
Hence $\Lambda _\Phi $ is contained in $\kappa _T({\bf R}^{2n})$ and
the two manifolds have the same dimension so they have to coincide
(near $(x_0,\xi _0)$).

\medskip
\noindent
We then know that
$$
{{\sigma }_\vert}_{\Lambda _\Phi }=\sum_1^n d\left(\frac{2}{i}\partial
_{x_j}\Phi (x)\right) \wedge
dx_j=\frac{2}{i}\sum_{k}\sum_j\partial _{\overline{x}_k}\partial
_{x_j}\Phi\, d\overline{x}_k\wedge dx_j
$$
is non-degenerate, so the Levi-form of $\Phi $ is non-degenerate. Since
$\Phi $ by definition is the supremum of the family of pluri-harmonic functions
$x\mapsto -\Im \phi (x,y)$ we know that $\Phi $ is pluri-subharmonic and hence
strictly pluri-subharmonic.
\end{proof}

\medskip
\noindent
For $y\in {\bf R}^n$ (close to $y_0$) let
$$\Gamma _y=\{ x\in {\bf
  C}^n; y(x)=y \}=\pi _x\kappa _T(T_y^*{\bf R}^n),$$
where $\pi _x:{\bf C}_{x,\xi }^{2n}\to {\bf C}_x^n$ is the natural projection,
 so that
$\Gamma _y$ is of real dimension $n$ and the $\Gamma _y$ form a
foliation of $\mathrm{neigh\,}(x_0,{\bf C}^n)$.
$\Gamma _y$ is totally real: $T_x\Gamma _y\cap iT_x\Gamma _y=0$,
$\forall x\in \Gamma _y$. In fact, $T_x\Gamma _y=\{ t_x\in {\bf
  C}^n;\, \phi ''_{yx}t_x\in {\bf R}^n \}$.

\medskip
\noindent
For every fixed real $y$:
\begin{equation}
\label{fbi.2}
\Phi (x)+\Im \phi (x,y)=-\Im \phi (x,y(x))+\Im \phi (x,y)\asymp
\mathrm{dist\,}(x,\Gamma _y)^2.
\end{equation}
Since $x\mapsto -\Im \phi (x,y)$ is pluri-harmonic, this gives another proof of
the fact that $\Phi (x)$ is strictly pluri-subharmonic.

\medskip
\noindent
{\bf Exercise} Explore the standard case of Bargmann transforms with $\phi (x,y)=i(x-y)^2/2$.

\medskip
\noindent
{\bf Exercise} Let $f(y)$ be analytic near $y_0$,
real valued on the real domain and with $f'(y_0)=\eta _0$. Show that
$$
T(e^{if/h})=h^{n/2}c(x;h)e^{ig(x)/h},
$$
where $c(x;h)$ is a classical analytic symbol of order $0$ and
$$g(x)=\mathrm{vc}_{y\in \mathrm{neigh\,}(y_0,{\bf C}^n)}(\phi
(x,y)+f(y))$$
is holomorphic, $\Lambda _g:=\{ (x,g'(x)) \}=\kappa
_T(\Lambda _f)$ where $\Lambda _f$ is defined as $\Lambda _g$.

\medskip
\noindent
Let $(\Lambda _f)_{\bf R}=\Lambda _f\cap {\bf R}^{2n}$. Show that
$-\Im g\le \Phi $ and that more precisely,
\begin{equation}
\label{fbi.29}
\Phi (x)+\Im g(x)\asymp \mathrm{dist\,}(x,\pi _x(\kappa _T((\Lambda_f)_{\bf R})))^2.
\end{equation}
Observe also that $\pi _x(\kappa _T((\Lambda _f)_{\bf R}))$ is transversal to $\Gamma _y$.

\medskip
\noindent
Assume that $\eta _0\ne 0$. For $x\in \mathrm{neigh\,}(x_0)$, write
$$
(y(x),\eta (x))=(y(x),-\partial _y\phi (x,y(x)))\in T^*{\bf
  R}^n\setminus 0,
$$
where $y(x)$ is the local real maximum of $-\Im \phi (x,\cdot
)$. Also, we have
$$
(y(x),\eta (x))=\kappa _T^{-1}(x,\frac{2}{i}\partial _x\Phi (x)).
$$
\begin{dref}\label{fbi3} Let $u$ be a distribution defined near
  $y_0$, independent of $h$. We say that $(y(x),\eta (x))\not\in
  \mathrm{WF}_a(u)$ if $Tu=0$ in $H_{\Phi ,x}$.
\end{dref}
We shall see that this defines a closed conic subset
$\mathrm{WF}_a(u)$ of $T^*(\mathrm{neigh\,}(y_0,{\bf R}^n))\setminus
0$, {\it independent } of the choice of $T$.

\medskip
\noindent
In order to prove that the definition does not depend on the
choice of $T$ we would
like to construct ``the inverse $T^{-1}$''. However, this can never
succeed completely since $Tu$ only carries microlocal information
about $u$ near $(y_0,\eta _0)$. We can however give meaning to this
inverse on certain smaller spaces and that will suffice to be able to
describe a second FBI-transform $\widetilde{T}u$ in terms of $Tu$.

\medskip
\noindent
Put
\begin{equation}\label{fbi.3}Sv(x;h)=h^{-n}\int e^{-i\phi (z,x)/h}b(z,x;h)v(z)dz,\end{equation}
where $b$ is an elliptic classical analytic symbol of order 0, defined
near $(x_0,y_0)$. Formally,
\begin{equation}\label{fbi.4}
STu(x;h)=h^{-n}\iint e^{i(-\phi (z,x)+\phi (z,y))/h}b(z,x;h)a(z,y;h)
u(y)dydz
\end{equation}
and we can apply the Kuranishi trick (change of variables in $z$) to
see that formally
\begin{equation}\label{fbi.5}
STu(x;h)=\frac{1}{(2\pi h)^n}\iint e^{\frac{i}{h}(x-y)\cdot \theta
}c(x,y,\theta ;h) u(y)dyd\theta ,
\end{equation}
where $c$ is an elliptic classical analytic symbol of order $0$,
defined near $(y_0,y_0,\eta _0)$. According to Lemma \ref{pf1} and the
previously given definition of the symbol of a pseudodifferential operator, we can replace
$c$ by $\widetilde{c}(x,\theta ;h)$, independent of $y$ and still elliptic
to get a new pseudodifferential operator which has the same action on expressions as in the
last exercise above.

\medskip
\noindent
Let $\widetilde{d}$ satisfy $\widetilde{d}\#
\widetilde{c}=1$. Then
$$
\widetilde{d}(x,hD_x;h)\circ ST =1
$$
when acting on functions as in the exercise. On the other hand we can
apply stationary phase to get formally
$$
\widetilde{d}(x,hD;h)Sv=h^{-n}\int e^{-i\phi /h }\widetilde{b}v(z)dz=:\widetilde{S}v(x;h)
$$
Our compositions are well defined and hence associative when restricted
to expressions as in the exercise and we therefore get
$$\widetilde{S}T=1.$$
Dropping the tildes, we have shown that we can find $S$ of
the form (\ref{fbi.3}) such that
$$ST=1$$
when acting on expressions as in the exercise.

\medskip
\noindent
When trying to define $Sv(x;h)$ for $v\in H_\Phi$, we would like to have a contour $\Gamma$ in $z$ space such that
$$
\Im \phi (z,x)+\Phi (z)\le 0,\ z\in \Gamma ,
$$
with strict inequality near the boundary. In view of (\ref{fbi.2}) the best possible choice in general is $\Gamma =\Gamma _x$ and we then just achieve equality.

\medskip
\noindent
If however $v\in H_\Psi $, where $\Psi -\Phi \asymp -\mathrm{dist\,}(z,\widetilde{\Gamma })^2$ and $\widetilde{\Gamma }$
is a real manifold of dimension $n$ transversal to $\Gamma _x$, then $Sv$ is well-defined. In particular if $u$ is as in the exercise, $v=Tu$, this
is the case with $\Psi =-\Im g$, so $Sv$ is well-defined up to an exponentially small ambiguity, and  we get $Sv\equiv u$ in $H_{-{\rm Im}\, f}$.

\medskip
\noindent
Let
$$\widetilde{T}u(x;h)=\int e^{i\widetilde{\phi
  }(x,y)/h}\widetilde{a}(x,y;h)u(y)dy$$ be a second FBI-transform with
$\widetilde{\phi }$, $\widetilde{a}$ defined near
$(\widetilde{x}_0,y_0)$ and with $-\widetilde{\phi }'_y(\widetilde{\xi
}_0,y_0)=\eta _0$. Then formally
\begin{equation}\label{fbi.6}
\widetilde{T}Sv(x;h)=h^{-n}\iint e^{\frac{i}{h}(\widetilde{\phi
  }(x,y)-\phi (z,y))}\widetilde{a}(x,y;h)a(z,y;h)u(y)dy dz.
\end{equation}
This is a Fourier integral operator\footnote{A general local theory
  for Fourier integral operators can be developed in the spirit of
  Section \ref{pf}. See \cite{Sj82}, Chapter 11.} with associated canonical transformation $\kappa _{\widetilde{T}}\circ \kappa
_T^{-1}$, mapping $\Lambda _\Phi $ to $\Lambda _{\widetilde{\Phi }}$
and it follows from this observation, or by direct verification, that
$$
(y,z)\mapsto -\Im \widetilde{\phi }(x,y)+\Im \phi (z,y)+\Phi (z)=:F
$$
has a non-degenerate critical point, given by the conditions
$$
(z,\frac{2}{i}\partial _z\Phi (z))=\kappa _T(y,\eta ),\
(x,\frac{2}{i}\partial _x\widetilde{\Phi } (x))=\kappa _{\widetilde{T}}(y,\eta ),
$$
where $(y,\eta )$ is real ($y=y(z)=\widetilde{y}(x)$, $\eta =\eta
(z)=\widetilde{\eta }(z) $).

\medskip
\noindent
Next, we show that there is a good contour for (\ref{fbi.6}):
As a first attempt, we take $y\in {\bf R}^n$, $z\in \Gamma _y$. Along
that contour we have
$$
F(y,z)-\widetilde{\Phi }(x)=-(\widetilde{\Phi }(x)+\Im \widetilde{\phi
}(x,y))\asymp -|y-\widetilde{y}(x)|^2.
$$
Thus the contour is ``almost good''. Since our critical point is
non-degenerate, it is then clear that we can make a small deformation
and find a good contour. In conclusion
$$\widetilde{T}S\hbox{ is a well-defined Fourier integral operator }H_{\Phi ,x_0}\to
H_{\widetilde{\Phi },\widetilde{x}_0}.$$
\begin{prop}\label{fbi4}
For $x\in \mathrm{neigh\,}(x_0)$,
$\widetilde{x}\in\mathrm{neigh\,}(\widetilde{x}_0)$ related by $$
\widetilde{\kappa }
_{\widetilde{ T}}^{-1}(\widetilde{ x},(2/i)\partial_{\widetilde{ x}}
\widetilde{\Phi } (\widetilde{ x}))=
\kappa
_T^{-1}(x,(2/i)\partial_x \Phi (x)),$$ the following two statements are equivalent:
\begin{itemize}
\item[1)] $\widetilde{T}u=0$ in $H_{\widetilde{\Phi },\widetilde{x}}$.
\item[2)] ${T}u=0$ in $H_{{\Phi },{x}}$.
\end{itemize}
\end{prop}
\begin{proof}
Take $x=x_0$, $\widetilde{x}=\widetilde{x}_0$ for simplicity. Let
$\chi \in C_0^\infty (\mathrm{neigh\,}(\eta _0,{\bf R}^n))$ be equal
to one near $\eta _0$. Without loss of generality, we may assume that
the distribution $u$ has compact support in a neighborhood of
$y_0$. Then from the (classical!) Fourier inversion formula,
$$
u(x)=\frac{1}{(2\pi h)^n }\int e^{ix\cdot \eta /h}{\cal F}u(\eta )d\eta ,
$$
and contour deformations, we see that
$$Tu=T\chi (hD_y)u\hbox{ in }H_{\Phi ,x_0},\quad
\widetilde{T}u=\widetilde{T}\chi (hD_y)u\hbox{ in }H_{\widetilde{\Phi } ,\widetilde{x}_0}.$$

\medskip
\noindent
On the other hand $v=\chi (hD_y)u$ is a superposition of plane waves
(special cases of states as in the last exercise), so
$$
\chi (hD_y)u=ST\chi (hD_y)u+{\cal O}(e^{-1/Ch}),
$$
where now
$$
Sv(y)=\int_{\Gamma _y}e^{-i\phi (x,y)/h}b(x,y;h)v(x)dx.
$$
Consequently,
$$
\widetilde{T}\chi (hD_y)u=\widetilde{T}\circ ST\chi (hD_y)u\hbox{ in
}H_{\widetilde{\Phi },\widetilde{x}_0}.
$$
Here, for each plane wave in $\chi (hD_y)u$, we can make a contour
deformation to the good contour  discussed above for the Fourier integral operator
$\widetilde{T}S$ and putting everything together, we get
$$
\widetilde{T}u=(\widetilde{T}S)(Tu)\hbox{ in }H_{\widetilde{\Phi },\widetilde{x}_0}.
$$
Since the Fourier integral operator $\widetilde{T}S$ maps $H_{\Phi ,x_0}\to
H_{\widetilde{\Phi },\widetilde{x}_0}$, we see that $\widetilde{T}u=0$
in $H_{\widetilde{\Phi },\widetilde{x}_0}$ if $Tu=0$ in
$H_{\widetilde{\Phi },\widetilde{x}_0}$. The converse
implication also holds.\end{proof}

\medskip
\noindent
This shows that the definition of $\mathrm{WF}_a(u)$ does not depend
on the choice of $T$. By a simple dilation in $h$ we then see that it
is a conic subset of $T^*X\setminus 0$ (if $X\subset {\bf R}^n$ is the open set where
$u$ is defined).
Another basic property of the analytic wavefront set is given by
\begin{prop}
\label{fbi5}
We have
$$
\pi _y(\mathrm{WF}_a(u))=\mathrm{Sing\, Supp}_a(u),
$$
where the right hand side denotes the analytic singular support,
i.e. the complement in $X$ of the largest open subset where $u$ is
real analytic.
\end{prop}
\medskip
\noindent
{\bf Idea of the proof.} We start by using a resolution of
the identity of the form $1=\int_{T^*{\bf R}^n} \pi _\alpha d\alpha $
where $\pi _\alpha $ is a Gaussian Fourier integral operator ``concentrated at $\alpha
$''. If $y_0\not\in \pi _y(\mathrm{WF}_a(u))$, then a simple
adaptation of the proof above shows that $\pi _\alpha u$ decays
exponentially when $\alpha_\eta  $ tends to infinity while $\alpha _y$
is confined to a small neighborhood of $y_0$. (Here we write $\alpha
=(\alpha _y,\alpha _\eta ) $.)

\section{Egorov's theorem and elliptic regularity.}
\setcounter{equation}{0}
Let $\widetilde{P}(y,D_y)=\sum_{|\alpha |\le m}a_\alpha (y)D_y^\alpha
$ be a differential operator with analytic coefficients, defined on an open set $X\subset
{\bf R}^n$. Let $T$ be an FBI-transform as
above. Then we have the Egorov theorem  which states that there exists
a pseudodifferential operator with classical analytic symbol, $P(x,hD_h;h):H_{\Phi,x_0}\to
H_{\Phi ,x_0}$ such that
$$
PTu=Th^m\widetilde{P}u \hbox{ in }H_{\Phi ,x_0}
$$
when $u\in {\cal D}'(X)$ is independent of $h$. Indeed, we can take
$P=Th^m\widetilde{P}S$. For the leading symbols, we have the relation
\begin{equation}\label{eg.1}
p\circ \kappa _T=\widetilde{p}.
\end{equation}
\begin{theo}\label{eg1}
In the above situation, let $u\in {\cal D}'(X)$ be independent of $h$
and assume that $\widetilde{P}u$ is analytic on $X$. Then
$\mathrm{WF}_a(u)\subset \widetilde{p}^{-1}(0)$.
\end{theo}
\begin{proof} Let $(y_0,\eta _0)\in T^*X\setminus 0$ be a
point where $\widetilde{p}(y_0,\eta _0)\ne 0$ and assume that
$(y_0,\eta _0)\not\in \mathrm{WF}_a(\widetilde{P}u)$ (which is a weaker assumption than in the
theorem). We choose $T$ adapted to the point $(y_0,\eta _0)$. Then
$$
PTu=0\hbox{ in }H_{\Phi ,x_0}\hbox{ and }p(x_0,\frac{2}{i}\partial
_x\Phi (x_0))\ne 0.
$$
Let $Q(x,\xi ;h)$ be a classical analytic symbol $Q\sim \sum_0^\infty
h^k q_k(x,\xi )$ such that
$$
Q\# P=1\hbox{ near }(x_0,\xi _0).
$$
Correspondingly, we have $Q(x,hD;h):H_{\Phi ,x_0}\to H_{\Phi ,x_0} $
so that
$$
Q(x,hD;h)\circ P(x,hD;h)=1: H_{\Phi ,x_0}\to H_{\Phi,x_0}.
$$
Apply this to $Tu$:
$$
Tu=QPTu=0\hbox{ in }H_{\Phi ,x_0}.
$$
Hence $(y_0,\eta _0)\not\in \mathrm{WF}_a(u)$. We have thus shown that
$\mathrm{WF}_a(u)\subset \mathrm{WF}_a(\widetilde{P}u)\cup
\widetilde{p}^{-1}(0)$ which is a stronger statement than in the
theorem.\end{proof}

\medskip
\noindent
For the notes of a course of more than 3 hours, it would here be the natural place to discuss the method of non-characteristic deformations and the
Kawai-Kashiwara theorem about propagation of analytic regularity for micro-hyperbolic operators. See \cite{Sj82}, Chapter 10.

\section{Analytic WKB and quasi-modes}
\setcounter{equation}{0}
\medskip
\noindent
Let $P(x,hD;h)$ be a classical analytic pseudodifferential operator of order $0$, defined
near $(0,\xi _0)\in {\bf C}^{2n}$, such that the leading symbol satisfies
$$
p(0,\xi _0)=0,\ \partial _{\xi _n}p(0,\xi _0)\ne 0.
$$
Let $\phi \in \mathrm{Hol\,}(\mathrm{neigh\,}(0,{\bf C}^n))$ solve the
eikonal problem
\begin{equation}\label{wkb.1}
p(x,\phi '(x))=0,\ \phi '(0)=\xi _0.
\end{equation}
Let $H$ be the hypersurface $x_n=0$. We use the standard notation
$x=(x',x_n)\in {\bf C}^n$.
\begin{theo}\label{wkb1}
Let $v(x;h)$, $w(x';h)$ be classical analytic symbols of order $0$
defined near $0$ in ${\bf C}^n$ and ${\bf C}^{n-1}$ respectively. Then
there exists a classical analytic symbol $u(x;h)$ defined near $0\in {\bf C}^n$ such that
\begin{equation}\label{wkb.2}
e^{-i\phi (x)/h}\circ P\circ e^{i\phi /h}u=hv,\ {{u}_\vert}_{H}=w.
\end{equation}
\end{theo}
\begin{proof} We may assume that $w=0$. Also $e^{-i\phi
  (x)/h}\circ P\circ e^{i\phi /h}$ is a classical analytic pseudodifferential operator of
order $0$ with leading symbol $p(x,\phi '_x(x)+\xi )$, so we may
assume that $\phi =0$, $p(x,0)=0$. After a change of variables, which
does not modify $H$, we may also assume that $\partial _{\xi
  '}p(x,0)=0$, $\partial _{\xi _n}p=i$, or in other words, $p(x,\xi
)=i\xi _n+{\cal O}(\xi ^2)$.

\medskip
\noindent
Writing $P=\sum_0^\infty h^kp_k(x,\xi )$, $p_0=p$, the first
equation in (\ref{wkb.2}) becomes
\begin{equation}\begin{split}\label{wkb.3}&\partial
    _{x_n}u+p_1(x,0)u(x;h)+\frac{1}{h}Au=v\\
&A=\sum_{k+|\alpha |\ge 2}\frac{h^k}{\alpha !}(\partial _\xi ^\alpha
p_k)(x,0)(hD_x)^\alpha =\sum_{k=2}^\infty h^kA_k,
\end{split}\end{equation}
where $A$ has the same general properties as in Section \ref{cas}. Assume for
simplicity that $p_1(x,0)=0$ (which otherwise can be achieved by conjugation).

\medskip
\noindent
Let $\Omega =\{ x\in {\bf C}^n;\, \frac{|x'|}{R}+\frac{|x_n|}{r}<1
\}$, where $R,r>0$ are small enough so that we stay in the domains of
definition of the various symbols and operators. For $0\le t\le r$, we
define $\Omega _t\subset {\bf C}^n$ by
$$
\frac{|x'|}{R-\frac{Rt}{r}}+\frac{|x_n|}{r-t}<1.
$$
Let $a\in \mathrm{Hol\,}(\Omega _0)$ have the property that for some
$k>1$:
$$
\sup_{\Omega _t}|a|\le C(a,k)t^{-k},\ 0<t\le r.
$$
Put
$$
\partial _{x_n}^{-1}a(x)=\int_0^{x_n}a(x',y_n)dy_n.
$$
Then
$$
\sup_{\Omega _t}|a|\le C(a,k)\int_t^{+\infty }s^{-k}ds = \frac{C(a,k)}{(k-1)t^{k-1}}.
$$

\medskip
\noindent
Let $a=\sum_2^\infty a_kh^k$ be a classical analytic symbol of order $-2$ such
that
\begin{equation}\label{wkb.4}
\sup_{\Omega _t}|a_k|\le \frac{f(a,k)k^k}{t^k},\ 0<t\le r,
\end{equation}
where $k\mapsto f(a,k)$ grows at most exponentially. Then,
$$b:=(h\partial _{x_n})^{-1}a=\sum_1^\infty b_kh^k,\ b_k=\partial
_{x_n}^{-1}a_{k+1},$$
$$
\sup_{\Omega _t}|b_k|\le \frac{f(a,k+1)(k+1)^{k+1}}{kt^k}\le 2ef(a,k+1)\frac{k^k}{t^k}.
$$
Hence, $f(b,k)\le 2e f(a,k+1)$, when defining $f(b,k)$ as in
(\ref{wkb.4}).

\medskip
\noindent
Put
$$
\| a\|_\rho=\sum_2^\infty f(a,k)\rho ^k,\  \| b\|_\rho=\sum_1^\infty
f(b,k)\rho ^k .
$$
Then
\begin{equation}\label{wkb.5}
\| b\|_\rho \le \frac{2e}{\rho }\| a\|_\rho .
\end{equation}
  The problem (\ref{wkb.2}), (\ref{wkb.3}), with $w=0$ and
  $p_1(x,0)=0$, can be written
\begin{equation}\label{wkb.6}
u+(h\partial _{x_n})^{-1}Au=h(h\partial _{x_n})^{-1}v=:\widetilde{v},
\end{equation}
where $\widetilde{v}$ is a classical analytic symbol of order
$0$. Defining $\| A\|_\rho $ as in Section \ref{cas} with respect to the
family $\Omega _t$, we have
$$
\| Au\|_\rho \le \| A \|_\rho \| u\|_\rho \le {\cal O}(\rho
^2)\| u\|_\rho ,
$$
when $\rho $ is small enough. Hence by (\ref{wkb.5}),
$$
\| (h\partial _{x_n})^{-1}Au\|_\rho \le {\cal O}(1)\rho \| u\|_\rho .
$$
We then see from (\ref{wkb.6}) that $\| u\|_\rho <\infty $ when $\rho
>0$ is small enough and we conclude that $u$ is an analytic symbol in
$\Omega _0$.\end{proof}

\medskip
\noindent
We next discuss {\it quasimodes for non-self-adjoint differential operators} in the semi-classical limit. Let
$$
P=P(x,hD_x;h)=\sum_{|\alpha |\le m}a_\alpha (x;h)(hD_x)^\alpha
$$
be a semi-classical differential operator defined on an open set
$\Omega \subset {\bf R}^n$. Assume that
\begin{equation}\label{wkb.7}
a_\alpha (x;h)\sim \sum_0^\infty a_\alpha ^k(x)h^k
\end{equation}
are (realizations of) classical analytic symbols. The semi-classical
principal symbol of $P$ is then
\begin{equation}\label{wkb.8}
p(x,\xi )=\sum_{|\alpha |\le m}a_\alpha (x)\xi ^\alpha .
\end{equation}
Let $(x_0,\xi _0)\in T^*\Omega $ be a point where
\begin{equation}\label{wkb.9}p(x_0,\xi _0)=0,\ \
  \frac{1}{2i}\{p,\overline{p} \} (x_0,\xi _0)>0.\end{equation}
Here, $\{ a,b \} =a_\xi '\cdot b'_x-a'_x\cdot b'_\xi $ denotes the
Poisson bracket of two sufficiently smooth functions $a(x,\xi )$,
$b(x,\xi )$. The following result, in a different non-semi-classical
formulation is due to H\"ormander \cite{Ho60a, Ho60b} in the smooth
setting and goes back to Sato-Kawai-Kashiwara \cite{SaKaKa71} in the analytic case. See \cite{DeSjZw04} for references and direct
proofs in the semi-classical formalism.
\begin{theo}\label{wkb2}
There exist an analytic function $\phi (x)$ and a classical analytic
symbol $b(x;h)$ of order 0, defined in a neighborhood of $x_0$ such
that
\begin{equation}\label{wkb.10}
\phi (x_0)=0,\ \ \phi '(x_0)=\xi _0,
\end{equation}
\begin{equation}\label{wkb.11}
p(x,\phi '(x))=0,\ x\in \mathrm{neigh\,}(x_0,\Omega ),
\end{equation}
\begin{equation}\label{wkb.12}
\Im \phi ''(x_0)>0,
\end{equation}
\begin{equation}\label{wkb.13}
P(\chi (x)b(x;h)e^{i\phi (x)/h})={\cal O}(1)e^{-\frac{1}{Ch}},\
C=C_\chi >0,
\end{equation}
if $\chi \in C_0^\infty (\mathrm{neigh\,}(x_0,\Omega ))$ is equal to
$1$ near $x_0$ and has its support sufficiently close to $x_0$,
\begin{equation}\label{wkb.14}
\| \chi b e^{i\phi /h}\|_{L^2}=h^{n/4}(1+{\cal O}(e^{-1/(Ch)})).
\end{equation}
\end{theo}

\medskip
\noindent
As usual, it follows from the proof that the conclusion remains
uniformly valid if we replace $P$ by $P-z$ for $z\in
\mathrm{neigh\,}(0,{\bf C})$. More generally the conclusion is valid
for $P-z$ for $z\in \mathrm{neigh\,}(z_0,{\bf C})$, if we replace the
condition $p(x_0,\xi _0)=0$ by $p(x_0,\xi _0)=z_0$ in (\ref{wkb.9}).

\medskip
\noindent
When $P$ can be realized as a closed operator on $L^2(\Omega )$
or on $L^2(M)$ for some manifold containing $\Omega $, then we
conclude that $\| (P-z)^{-1}\| \ge e^{1/(Ch)}/C$ for some $C>0$ and
for $z\in \mathrm{neigh\,}(z_0,{\bf C})\setminus \sigma (P)$, where
$\sigma (P)$ denotes the spectrum of $P$. Notice that $i^{-1}\{
p,\overline{p} \}$ is the semi-classical principal symbol of the
commutator $h^{-1}[P,P^*]$, so $P$ is non-normal.

\medskip
\noindent
When $P$ is a fixed elliptic operator in the classical sense, with analytic $h$-indepen\-dent coefficients, the result with some
obvious modifications applies to $P-z$ when $z$ tends to infinity in a narrow sector.

\medskip
\noindent
We refer to \cite{DeSjZw04} for a fuller discussion of the spectral aspects.

\medskip
\noindent
{\bf Proof of Theorem \ref{wkb2}.} The assumption (\ref{wkb.9}) implies that
$p'_\xi (x_0,\xi _0)\ne 0$. The existence of analytic solutions to (\ref{wkb.10}), (\ref{wkb.11}) then follows from complex Hamilton-Jacobi theory
or simply from the Cauchy-Kowalevska theorem. More precisely, if $H$ is a complex hypersurface in $x$-space that passes through
$x_0$ transversally to $p'_\xi (x_0,\xi _0)\cdot \partial _x$ and $\psi$ is holomorphic on $\mathrm{neigh\,}(x_0,H)$ with
$d\psi ={{\xi_0\cdot dx }_\vert}_{H}$ at $x_0$, then (\ref{wkb.10}), (\ref{wkb.11}) has a solution $\phi $ such that
${{\phi}_\vert}_{H}=\psi $, unique near $x_0$.

\medskip
\noindent
For (\ref{wkb.12}) we recall a geometric characterization by
H\"ormander \cite{Ho71b}. Let $\Lambda _\phi $ be the complex
Lagrangian manifold defined near $(x_0,\xi _0)$ by $\xi =\phi '(x)$
where $\phi (x)$ is holomorphic near $x_0$ and $\phi '(x_0)=\xi_0$. Then,
\begin{itemize}
\item (\ref{wkb.12}) $\Longrightarrow$
\begin{equation}\label{wkb.15}
\frac{1}{i}\sigma (t,\overline{t})>0,\ \forall t\in T_{x_0,\xi
  _0}(\Lambda _\phi )\setminus \{ 0 \},
\end{equation}
where we view the symplectic form $\sigma $ as an alternate bilinear
form.
\item If $\Lambda $ is a complex Lagrangian manifold containing
  $(x_0,\xi _0)$ such that (\ref{wkb.15}) holds, then after
  restricting $\Lambda $ to a small neighborhood of $(x_0,\xi _0)$, we
  get $\Lambda =\Lambda _\phi $, where $\phi $ is holomorphic near
  $x_0$ and satisfies (\ref{wkb.10}), (\ref{wkb.12}).
\end{itemize}

\medskip
\noindent
The geometric formulation of the problem  (\ref{wkb.10})--(\ref{wkb.12}) is then to find a complex Lagrangian manifold
$\Lambda \subset \Gamma:=p^{-1}(0)$ which contains $(x_0,\xi _0)$ and is strictly positive in the sense of (\ref{wkb.15}). Notice that the strict positivity of
$\Lambda $ at $(x_0,\xi _0)$ implies that $\Lambda $ intersects $T^*\Omega $ transversally at $(x_0,\xi _0)$.

\medskip
\noindent
Here $\Gamma =p^{-1}(0)$ denotes the {\it complex} hypersurface and we recall that $H_p$ is tangent to $\Gamma $. We also know by
elementary symplectic geometry that $H_p$ is everywhere tangent to $\Lambda $.

\medskip
\noindent
Let $\Sigma =p^{-1}(0)\cap \mathrm{neigh\,}((x_0,\xi
_0),T^*\Omega )$ be the real characteristic manifold. It is symplectic
and of codimension 2. Let $\Sigma ^{\bf C}\subset
\mathrm{neigh\,}((x_0,\xi _0),{\bf C}^{2n})$ denote its
complexification. It is a complex symplectic manifold of codimension 2
in ${\bf C}^{2n}$, given by the equations $p(\rho )=0$, $p^*(\rho
)=0$, where $p^*(\rho )=\overline{p(\overline{\rho })}$. The assumption
(\ref{wkb.9}) implies that $\Sigma ^{\bf C}$ is a complex hypersurface
in $\Gamma $, given there by the equation $p^*(\rho )=0$, transversal
to $H_p$ since $H_pp^*=\{ p ,\overline{p} \}\ne 0$.

\medskip
\noindent
It is now clear that the complex Lagrangian manifolds $\Lambda$ with
$(x_0,\xi _0)\in \Lambda \subset \mathrm{neigh\,}((x_0,\xi_0),\Gamma )$ coincide near that point with the ones of the form
$$
\{ \exp (zH_p)(\rho ');\, \rho'\in \Lambda ',\ z\in D(0,\varepsilon )\},
$$
where $\varepsilon >0$ is small and $\Lambda '$ is a complex Lagrangian submanifold of
$\Sigma ^{\bf C}$ containing $(x_0,\xi _0)$. By the Darboux theorem, $\Sigma $, $\Sigma ^{\bf C}$ can locally be identified with
${\bf R}^{2(n-1)}$, ${\bf C}^{2(n-1)}$, and we see that $\Lambda $ is strictly positive at $(x_0,\xi _0)$ iff $\Lambda '$ is. Indeed, a general
$t\in T_{(x_0,\xi_0)}\Lambda $ is of the form $t=t'+zH_p(x_0,\xi _0)$, for $t'\in
T_{(x_0,\xi _0)}\Lambda '$, $z\in {\bf C}$ and since $\sigma
(t',H_p)=\sigma (t',\overline{H_p})=0$, we get
\[
\begin{split}
&\frac{1}{2i}\sigma (t,\overline{t})= \frac{1}{2i}\sigma
(t',\overline{t'})+ \frac{|z|^2}{2i}\sigma (H_p,\overline{H_p})\\
&=\frac{1}{2i}\sigma
(t',\overline{t'})+ \frac{|z|^2}{2i}\{p,\overline{p} \}\asymp
|t'|^2+|z|^2\asymp |t|^2. \end{split}
\]

\medskip
\noindent
Now there are plenty of strictly positive Lagrange manifolds $\Lambda '\subset \Sigma ^{\bf C}$ passing through $(x_0,\xi _0)$ and
hence there are plenty of strictly positive Lagrange manifolds $\Lambda \subset \Gamma $ containing that point. This means that
we have plenty of solutions to the problem (\ref{wkb.10})--(\ref{wkb.12}).

\medskip
\noindent
We choose one such solution $\phi (x)$ and apply Theorem \ref{wkb1} to conclude that there
exists an elliptic classical analytic symbol $b(x;h)\sim \sum_0^\infty
b_k(x)h^k$ such that formally,
$$
P(x,hD;h)(b(x;h)e^{i\phi (x)/h})=0,\ x\in \mathrm{neigh\,}(x_0,\Omega ).
$$
This means that (if $b$ also denotes a realization as in Theorem \ref{wkb2})
$$
P(x,hD_x;h)(be^{i\phi /h})={\cal O}(e^{-1/(Ch)})e^{i\phi /h}.
$$
From (\ref{wkb.12}) we see that $e^{i\phi (x)/h}$ is exponentially decaying on the real domain away from any fixed neighborhood of
$x_0$. Thus, if $\chi $ is a cutoff as in the statement of the theorem,
$$
P(\chi be^{i\phi /h})={\cal O}(e^{-1/(Ch)}).
$$
By analytic stationary phase,
$$
\| \chi be^{i\phi /h}\|^2_{L^2}=h^{\frac{n}{2}}c(h),
$$
where $c(h)\sim c_0+c_1h+...$ is a positive elliptic analytic symbol. Applying the quasinorms of Section \ref{cas} (that simplify a
lot since the family $\Omega _t$ is absent), we see that $c^{-1/2}$ is a classical analytic symbol. Replacing $b$ with $c^{-1/2}b$, we get
(\ref{wkb.13}), (\ref{wkb.14}).

\section{Propagation of regularity along a real bi\-cha\-ra\-cteris\-tic strip}
\label{pr}
\setcounter{equation}{0}
Let $P$ be a differential operator with analytic coefficients on an open set $X\subset {\bf R}^n$. Let $p$ be the principal symbol. The following theorem is
due to N.~Hanges \cite{Ha81}. It improves the classical propagation theorem of L.~H\"ormander \cite{Ho71c} and Sato, Kawai and Kashiwara
\cite{SaKaKa71} for operators of real principal type in that it only requires one real bicharacteristic strip. See also \cite{HaSj82}.
\begin{theo}
\label{pr1}
Assume that $H_p=p'_\xi \cdot \partial _x-p'_x\cdot \partial _\xi $
has a real integral curve $\gamma :[a,b]\to p^{-1}(0)\cap
T^*X\setminus 0$, $a<b$. If $u\in {\cal D}'(X)$,
$\mathrm{WF}_a(Pu)\cap \gamma ([a,b])=\emptyset $, then $\gamma
([a,b])$ is either contained in, or disjoint from $\mathrm{WF}_a(u)$.
\end{theo}
The proof uses a WKB-construction and the variant we give here is slightly different from the one in Chapter 9 in \cite{Sj82}.

\medskip
\noindent
If $dp$ vanishes at some point of $\gamma $, then $\gamma$ is reduced to a point and the statement in the theorem becomes
trivial. Hence, we may assume that $dp\ne 0$ along $\gamma $.
\begin{theo}
\label{pr2}
Assume that $p(y_0,\eta _0)=0$, $dp(y_0,\eta _0)\ne 0$. Then we can
find an FBI-transform $T$ defined near $(y_0,\eta _0)$ such that
$hD_{x_n}Tu=Th^mPu$ in $H_{\Phi ,x_0}$, for $u\in {\cal D}'(X)$
independent of $h$.
\end{theo}
\begin{proof} We start with the phase.
\begin{lemma}\label{pr3}
There exists an FBI-phase $\phi (x,y)$, defined near $(x_0,y_0)$ such
that
\begin{equation}\label{pr.1}
\partial _{x_n}\phi =p(y,-\partial _y\phi (y)).
\end{equation}
\end{lemma}
\begin{proof} We put
$$
\phi (x',0,y)=\frac{i}{2}(x'-y')^2-\eta _{0,n}y_n+iC(y_n-y_{0})^2,
$$
and choose $x_0=(y_0'-i\eta _0',0)$. Here $C$ will be chosen with $\Re
C>0$. Then $\phi _y'((x_0',0),y_0)=-\eta _0$ and we let $\phi (x,y)$
be the corresponding local solution of (\ref{pr.1}).
Then $\phi $ fulfills the first two conditions in (\ref{fbi.1}). In
order to have $\det \phi ''_{xy}(x_0,y_0)\ne 0$, we may assume, after
a change of coordinates in $y$, that
$$
\partial _{\eta _n}p(y_0,\eta _0)\ne 0, \hbox{ or }[ \partial _\eta
p(y_0,\eta _0)=0 \hbox{ and }\partial _{y_n}p(y_0,\eta _0)\ne 0.]
$$
Then we can find $C$ with $\Re C>0$ such that
\begin{equation}\label{pr.2}
\partial _{y_n}(p(y,-\partial _y\phi ))\ne 0\hbox{ at }(x_0,y_0).
\end{equation}
Now the following statements are equivalent:
\begin{itemize}
\item $\det \phi ''_{xy}(x_0,y_0)\ne 0$,
\item $y\mapsto \partial _x\phi $ has bijective differential at
  $x=x_0$, $y=y_0$,
\item $y\mapsto (\partial _{x'}\phi ,p(y,-\partial _\eta \phi ))$ has bijective differential at
  $x=x_0$, $y=y_0$,
\item (\ref{pr.2}).
\end{itemize}
The last equivalence follows from
$$
\det \phi ''_{x',y'}\ne 0,\ \phi ''_{y_n,x'}=0\hbox{ at }(x_0,y_0).
$$
Thus $\phi $ is an FBI-phase.\end{proof}

\medskip
\noindent
We can now finish the proof of the last theorem. Take $\phi $ as in
the lemma. It suffices to choose $a$ in (\ref{fbi.1.5}) such that
$$
(hD_{x_n}-h^mP^{\mathrm{t}}(y,D_y))\left( e^{i\phi (x,y)/h} a(x,y;h)
\right) =0,
$$
which we can solve locally as in the preceding section with a
prescribed $a(x',0,y;h)$.\end{proof}

\medskip
\noindent
{\bf Proof of Hanges' theorem}: We may
decompose $[a,b]$ into finitely many short intervals, each being
covered by one FBI transform. Thus we may assume that $\gamma ([a,b])$ is
contained in a small neighborhood of $(y_0,\eta _0)$. Let $T$ be a
corresponding FBI transform as in the last theorem. Then $\kappa
_T\circ \gamma $ is an integral curve in $\Lambda _\Phi $ of $H_{\xi
  _n}=\partial _{x_n}$ on which $\xi _n$ vanishes. Assume for
simplicity that $x_0=0$. Then we know that
$$\frac{2}{i}\partial _x\Phi (0,t)=\xi _0=(\xi _0',0)$$
and consequently $\Phi (x)=-\Im (x'\cdot \xi _0')+{\cal O}(x'^2)$.

\medskip
\noindent
By the intertwining property and the fact that $\gamma ([a,b])$ is disjoint from $\mathrm{WF}_a(Pu)$, we know that
$$
hD_{x_n}Tu=0 \hbox{ in } H_\Phi (\mathrm{neigh}(\{0 \}\times [a,b],{\bf C}^n)),
$$
so by integration,
$$
Tu=v(x')+{\cal O}(e^{-{\rm Im}\, (x'\cdot \xi _0')/h-\epsilon /h}) \hbox{
  near }\{0 \}\times [a,b].
$$
Consequently, if $Tu=0$ in $H_{\Phi ,\gamma (t)}$ for some $t\in
[a,b]$ we have the same fact for all $t\in [a,b]$. In other words, if
$\gamma (t)\not\in \mathrm{WF}_a(u)$ for some $t\in [a,b]$, the same
must hold for all $t\in [a,b]$.

\section{Related results and developments}
\label{rel}
\setcounter{equation}{0}
The work \cite{Sj82} was the natural continuation of a series of works
on the propagation of singularities for solutions of boundary value
problems of order 2 and higher in the analytic category, \cite{Sj80a,
  Sj80b, Sj80c, Sj81a, Sj81, RaSj81} In the case of second order
operators, the main result here is that the analytic wavefront set for
solutions to homogenous problems is a union of maximally extended
analytic rays (and a more general microhyperbolic propagation theorem
for operators of higher order). This is analogous to the corresponding
result in the $C^\infty $ by M.~Taylor, R.~Melrose, G.~Eskin,
V.~Ivrii, culminating in \cite{MeSj78, MeSj82}, stating that the
ordinary $C^\infty $ wavefront set of solutions to the homogeneous
problem is a union of maximally extended $C^\infty $-rays. Such rays
have (with the exception of some slightly pathological cases) unique
extensions while analytic rays have non-unique extensions from points
where they are tangential to the boundary and the domain is concave in
the ray direction so that the complement, that we may call ``the
obstacle'', is convex in the same direction. Roughly, analytic rays
may glide along the boundary into the $C^\infty $ shadow region.

\medskip
\noindent
The methods used another kind of FBI-transforms, closely related to Gaussian resolutions of the identity. In \cite{Sj82} such resolutions
still play a role, while in the present text, we have eliminated them completely. It would have been nice if there had been time and energy
to revisit the boundary propagation in \cite{Sj82} with the improved methods there.

\medskip
\noindent
G.~Lebeau \cite{Leb84} explored the propagation of singularities for the wave equation outside a strictly convex obstacle in the whole
scale of Gevrey spaces $G^s$ that interpolate between the smooth and the analytic functions and found that the essential difference between
the two types of propagations appears at the value $s=3$. See also \cite{LaLa91}.

\medskip
\noindent
A related area is that of analytic hypoellipticity for non-elliptic operators. Here F.~Treves \cite{Tr78} and later
D.~Tartakoff \cite{Ta80} established analytic hypoellipticity for operators of the type $\Box_b$ that degenerate to order 2 on a
symplectic submanifold of the real cotangent space. The approach of Treves is based on a full fledged
machinery of analytic pseudodifferential operators and reductions to model-like cases while the one of Tartakoff is restricted to a more
special class of operators and uses very sophisticated iterated a priori-estimates to gain control of high order derivatives
directly. G.~M\'etivier \cite{Met81} in a still very long paper generalized the results to operators with multiple characteristics
following the general approach of Treves.

\medskip
\noindent
In \cite{Sj83} the second author gave a short proof of M\'etivier's result as well as some generalizations. We refer to \cite{GrScSj81,
  GrSj85} for some related results. The method of \cite{Sj83} is that of subelliptic deformations: After an FBI-transform we work in a space
$H_{\Phi _0}^\mathrm{loc}$ for some strictly plurisubharmonic weight $\Phi _0$ and the given subelliptic operator satisfies an
a priori-estimate in that space. We then look for a small deformation $\Phi \approx \Phi _0$ such that $P$ satisfies a nice a priori
estimate also in $H_\Phi ^\mathrm{loc}$ and such that $\Phi <\Phi _0$ where we want to obtain analytic regularity and $\Phi \ge \Phi _0$
near the boundary of a neighborhood of those points. A variant of the method used when we have micro-hyperbolicity, is to make deformations
such that the operator on the FBI-side is elliptic on $\Lambda _\Phi$, $\Phi >\Phi _0$ in a region where we want to gain analytic
regularity and such that on the boundary of a slightly larger region we have that $\Phi >\Phi _0$ only at points where already have
analytic regularity by assumption. The deformation of weights on the FBI-side corresponds to a local deformation $\kappa _T^{-1}(\Lambda
_\Phi )$ of the real phase space $T^*\Omega $ (locally equal to $\kappa _T^{-1}(\Lambda _{\Phi _0})$). See \cite{Sj82, Sj80c}.

\medskip
\noindent
In the theory of scattering poles (resonances) and other branches of spectral theory for non-self-adjoint (pseudo-)differential operators,
many works rely on phase space deformations which are now global. Since this activity started later we simply refer to some of
the works which also include some of those devoted to other global questions: \cite{44}--\cite{BoFuRaZe11}.


\begin{thebibliography}{40}

\bibitem{CaGrHiSj} E. Caliceti, S. Graffi, M. Hitrik, J. Sj\"ostrand, {\it Quadratic ${\cal PT}$--symmetric operators with real spectrum and similarity
to self-adjoint operators}, J. Phys. A: Math. Theor., {\bf 45} (2012), 444007.

\bibitem{CF} A. Cordoba and C. Fefferman, {\it Wave packets and Fourier integral operators}, Comm. PDE {\bf 3} (1978), 979–-1005.

\bibitem{DeSjZw} N. Dencker, J. Sj\"ostrand, and M. Zworski, {\it Pseudo-\-spectra of se\-mi\-clas\-si\-cal
(pseudo)\-diffe\-ren\-tial ope\-rators}, Comm. Pure Appl. Math. {\bf 57} (2004), 384-415.

\bibitem{DSj} M. Dimassi and J. Sj\"ostrand, {\it Spectral asymptotics in the semi-classical limit}, Cambridge University Press, 1999.

\bibitem{Hi04} M. Hitrik, {\it Boundary spectral behavior for semiclassical operators in dimension one}, Int. Math. Res. Not. {\bf 64} (2004), 3417--3438.

\bibitem{HiSj1} M. Hitrik and J. Sj\"ostrand, {\it Non-selfadjoint perturbations of selfadjoint operators in {\rm 2} dimensions {\rm I}},
Ann. Henri Poincar\'e {\bf 5} (2004), 1--73.

\bibitem{HiSjVu07} M. Hitrik, J. Sj\"ostrand, and S. V\~u Ng\d{o}c, {\it Diophantine tori and spectral asymptotics for non-selfadjoint operators},
Amer. J. Math. {\bf 129} (2007), 105--182.

\bibitem{HiSj15} M. Hitrik and J. Sj\"ostrand, {\it Rational invariant tori and band edge spectra for non-selfadjoint operators}, preprint, 2015.

\bibitem{H71} L. H\"ormander, {\it On the existence and the regularity of solutions of linear pseudo-differential equations}, Ens. Math. {\bf 17}
(1971), 99–-163.

\bibitem{HormI} L. H\"ormander, {\it The Analysis of Linear Partial Differential Operators {\rm I}}, Springer Verlag, 1990.

\bibitem{H91} L. H\"ormander, {\it Quadratic hyperbolic operators}, Microlocal Analysis and Applications
(Cattabriga, L. and Rodino, L., eds.), pp. 118--160, Lecture Notes in Math. 1495, Springer-Verlag, Berlin--Heidelberg, 1991.

\bibitem{KSjU} C. Kenig, J. Sj\"ostrand, and G. Uhlmann, {\it The Calder\'on problem with partial data}, Ann. of Math., {\bf 165} (2007), 567--591.

\bibitem{MeSj1} A. Melin and J. Sj\"ostrand,
{\it Determinats of pseudodifferential operators and complex deformations of phase space}, Methods and Appl. of Analysis {\bf 9} (2002), 177--238.

\bibitem{MeSj2} A. Melin and J. Sj\"ostrand, {\it Bohr-Sommerfeld quantization condition for non-selfadjoint operators
in dimension {\rm 2}}, Ast\'erisque {\bf 284} (2003), 181--244.

\bibitem{SjAst} J. Sj\"ostrand, {\it Singularit\'es analytiques microlocales}, Ast\'erisque, {\bf 95} (1982), 1--166,
Soc. Math. France, Paris.

\bibitem{Sj90} J. Sj\"ostrand, {\it Geometric bounds on the density of resonances for semiclassical problems}, Duke Math. J., {\bf 60} (1990), 1--57.

\bibitem{Sj95} J. Sj\"ostrand, {\it Function spaces associated to global I-Lagrangian manifolds}, Structure of solutions of differential
equations, Katata/Kyoto, 1995, World Sci. Publ., River Edge, NJ (1996).

\bibitem{Sj02} J. Sj\"ostrand, {\it Lectures on resonances}, lecture notes, 2002, {\sf http://sjostrand.perso.math.cnrs.fr/}

\bibitem{Zworski} M. Zworski, {\it Semiclassical analysis}, American Math. Society, 2012.



\end{thebibliography}

\begin{thebibliography}{30}
\bibitem{An70} K.G.~Andersson, {\it Propagation of analyticity of
    solutions of partial differential equations with constant
    coefficients,} Ark. f. Matematik. 8(1970), 277--302.
\bibitem{Bo64} J.~Boman, {\it On the intersection of classes of infinitely differentiable functions,} Ark. f. Matematik, 5 (1964), 301--309.
\bibitem{BoKr67} L.~Boutet de Monvel, P.~Kr\'ee, {\it
    Pseudo-differential operators and Gevrey classes,}
  Ann. Inst. Fourier (Grenoble) 17(1)(1967), 295--323.
\bibitem{BrIa75}J.~Bros, D.~Iagolnitzer, {\it Tubo{\"\i}des et
    structure analytique des distributions. II. Support essentiel et
    structure analytique des distributions,} (French) S\'eminaire
  Goulaouic-Lions-Schwartz 1974--1975: \'Equations aux d\'eriv\'ees
  partielles lin\'eaires et non lin\'eaires, Exp. No. 18, 34
  pp. Centre Math., \'Ecole Polytech., Paris, 1975.
\bibitem{DeSjZw04} N.~Dencker, J.~Sj\"ostrand, M.~Zworski,  {\it
    Pseudospectra of semiclassical (pseudo-) differential operators,}
  Comm. Pure Appl. Math. 57(3)(2004), 384--415.
\bibitem{Ehr60} L.~Ehrenpreis, {\it Solutions of some problems of
    division IV. Invertible and elliptic operators,} Amer. J. Math.82,
  522--588 (1960).
\bibitem{GrSj85} A.~Grigis, J.~Sj\"ostrand, {\it Front d'onde
    analytique et sommes de carrés de champs de vecteurs,} Duke
  Math. J. 52(1)(1985), 35--51.
\bibitem{GrScSj81} A.~Grigis, P.~Schapira, J.~Sj\"ostrand, {\it
    Propagation de singularités analytiques pour des opérateurs à
    caractéristiques multiples,} C. R. Acad. Sci. Paris Sér. I
  Math. 293(8)(1981), 397--400.
\bibitem{Ha81} N.~Hanges, {\it Propagation of analyticity along real
    bicharacteristics,}
Duke Math. J. 48(1)(1981), 269--277.
\bibitem{HaSj82} N.~Hanges, J.~Sj\"ostrand, {\it Propagation of
    analyticity for a class of non-micro-characteristic operators,}
  Ann. Math. 116(1982), 559-577.
\bibitem{Ho60a} L.~H\"ormander, {\it Differential equations without
    solutions,} Math. Ann.  140(1960), 169--173.
\bibitem{Ho60b} L.~H\"ormander, {\it Differential operators of
    principal type,} Math. Ann. 140(1960), 124--146.
\bibitem{Ho71c} L.~H\"ormander, {\it Uniqueness theorems and wave
    front sets for soutions of linear partial differential equations
    with analytic coefficients,} Comm. Pure Appl. Math. 24(1971),
  671--704.
\bibitem{Ho71b} L.~H\"ormander, {\it On the existence and the
    regularity of solutions of linear pseudo-differential equations,}
  S\'erie des Conf\'erences de l'Union Math\'ematique Internationale,
  No. 1. Monographie No. 18 de l'Enseignement
  Math\'ematique. Secr\'etariat de l'Enseignement Math\'ematique,
  Universit\'e de Gen\`eve, Geneva, 1971. 69 pp.
\bibitem{IaSt77} D.~Iagolnitzer, H.P.~Stapp, {\it The
    pole-factorization theorem in $S$-matrix theory,}
  Comm. Math. Phys. 57(1)(1977), 1--30.
\bibitem{IaSt69} D.~Iagolnitzer, H.P.~Stapp, {\it Macroscopic
    causality and physical region analyticity in $S$-matrix theory,}
  Comm. Math. Phys. 14(1969), 15--55.
\bibitem{LaLa91} B.~Lascar, R.~Lascar, {\it Propagation des
    singularit\'es Gevrey pour la diffraction,} Comm. Partial
  Differential Equations 16(4--5)(1991), 547--584.
\bibitem{Leb84} G.~Lebeau, {\it R\'egularit\'e Gevrey $3$ pour la
    diffraction,} Comm. Partial Differential Equations 9(15)(1984),
  1437--1494.
\bibitem{Ma42} S.~Mandelbrojt, {\it Analytic functions and classes of
    infinitely differentiable functions,} Rice Inst. Pamphlet
  No. 29:1, 1942.
\bibitem{Ma52} S.~Mandelbrojt, {\it S\'eries
    adh\'erentes. R\'egularisation des suites. Applications,}
  Gauthier-Villars, 1952.
\bibitem{Ma02a}A.~Martinez, {\it An introduction to semiclassical and
    microlocal analysis,} Universitext. Springer-Verlag, New York,
  2002.
\bibitem{MeSj78} R.~Melrose, J.~Sj\"ostrand, {\it Singularities of
    boundary value problems I,} CPAM, 31(5)(1978), 593-617.
\bibitem{MeSj82} R.~Melrose, J.~Sj\"ostrand, {\it Singularities of
boundary value problems II,}  CPAM, 35(1982),
129-168.
\bibitem{Met81} G.~M\'etivier, {\it Analytic hypoellipticity for
    operators with multiple characteristics,} Comm. Partial
    Differential Equations 6(1)(1981), 1--90.

\bibitem{RaSj81} J.~Rauch, J.~Sj\"ostrand, {\it Propagation of analytic
    singularities along diffracted rays} Indiana Univ. Math. J.,
    30(3)(1981), 283-401.

  \bibitem{SaKaKa71} M.~Sato, T.~Kawai, M.~Kashiwara, {\it
      Microfunctions and pseudo-differential equations,}
    Hyperfunctions and pseudo-differential equations (Proc. Conf.,
    Katata, 1971), pp. 265--529. Lecture Notes in Math., Vol. 287,
    Springer, Berlin, 1973.

\bibitem{Sj80a} J.~Sj\"ostrand, {\it Propagation of
analytic singularities for second order Dirichlet
problems,} Comm. PDE, 5(1)(1980), 41-94.

\bibitem{Sj80b} J.~Sj\"ostrand, {\it Propagation of analytic
singularities for second order Dirichlet   problems II,} Comm. PDE,
5(2)(1980), 187-207.

\bibitem{Sj80c} J.~Sj\"ostrand,
{\it Analytic singularities and microhyperbolic boundary value problems,}
Math. Ann., 254(1980), 211-256.

\bibitem{Sj81} J.~Sj\"ostrand, {\it Analytic
singularities of solutions of boundary value problems,} in
''Singularities in Boundary value problems'', Reidel publ.Co.(1981),
 235-269.
\bibitem{Sj81a} J.~Sj\"ostrand, {\it Propagation of
analytic singularities for second order Dirichlet   problems
III,} Comm. PDE, 6(5)(1981), 499-567.
\bibitem{Sj82} J.~Sj\"ostrand, {\it Singularit{\'e}s analytiques
    microlocales,} Ast{\'e}risque, 95(1982).
\bibitem{Sj83} J.~Sj\"ostrand, {\it Analytic wavefront sets and
    operators with multiple characteristics,} Hokkaido Math. J. 12
  (1983), no. 3, part 2, 392--433.
\bibitem{Ta80} D.S.~Tartakoff, {\it The local real analyticity of
    solutions to $\square_{b}$ and the $\bar \partial $-Neumann
    problem,} Acta Math. 145(3--4)(1980), 177--204.
\bibitem{Tr78} F.~Tr\`eves, {\it Analytic hypo-ellipticity of a class
    of pseudodifferential operators with double characteristics and
    applications to the $\overline\partial$-Neumann problem,}
  Comm. Partial Differential Equations 3(6--7)(1978), 475--642.
\bibitem{Tr80a}F.~Treves, {\it Introduction to pseudodifferential and
    Fourier integral operators. Vol. 1. Pseudodifferential operators,}
  The University Series in Mathematics. Plenum Press, New York-London,
  1980.

\bibitem{44} B.~Helffer, J.~Sj\"ostrand,
{\it R{\'e}sonances en limite semiclassique,} Bull. de la SMF   114(3),
M{\'e}moire 24/25(1986).


\bibitem{49} C.~G{\'e}rard, J.~Sj\"ostrand, {\it Semiclassical
resonances generated by a closed   trajectory of hyperbolic
type,} Comm. Math.Phys.,108(1987), 391-421.

\bibitem{50}
C.~G{\'e}rard, J.~Sj\"ostrand, {\it R{\'e}sonances en limite semiclassique et
exposants de   Lyapunov,} Comm. Math. Phys. 116(1988),
193-213.

\bibitem{51} J.~Sj\"ostrand, {\it Semiclassical resonances generated
by a non-degenerate critical   point,} Springer LNM, 1256,
402-429.

\bibitem{53} J.~Sj\"ostrand, {\it Estimates on the number of
resonances for semiclassical Schr{\"o}dinger   operators,} Proceedings
of the 8:th Latin-American School of Mathe\-ma\-tics, 1986,
Springer LNM , 1324 (1988), 286-292.

\bibitem{62}  J.~Sj\"ostrand, {\it Geometric bounds on the density
of resonances for semiclassical problems,} Duke Mathematical
Journal, 60(1)(1990), 1-57.

\bibitem{57}
B.~Helffer, J.~Sj\"ostrand, {\it Semiclassical analysis for Harper's equation III.
Cantor Structure of the spectrum,} Bull. de la SMF 117(4)(1989),
 m{\'e}moire no 39.


\bibitem{97} J.~Sj\"ostrand, {\it Density of resonances for strictly
convex analytic obstacles,} Can. J. Math., 48(2)(1996),
397-447.


\bibitem{100} J.~Sj\"ostrand, {\it Function spaces associated to
global I-Lagrangian manifolds,} pages 369-423 in Structure of
solutions of differential equations, Katata/Kyoto, 1995, World
Scientific 1996


\bibitem{114} J.~Sj\"ostrand, {\it Quantum
resonances and trapped trajectories,} pages 33--61, in Long Time
behaviour of classical and quantum systems, proc. Bologna
APTEX Internat. Conf., 13--17 September 1999, Series on Concrete
and Applicable Math., Vol 1, World Scientific,
2001.

\bibitem{116}
A.~Melin, J.~Sj\"ostrand, {\it Determinants of pseudodifferential operators
and complex deformations of phase space,}   Methods and Applications
of Analysis,  9(2)(2002), 177-238.
\bibitem{117}
A.~Melin, J.~Sj\"ostrand, {\it Bohr-Sommerfeld quantization condition
for non-selfadjoint operators in dimension 2,}  Ast{\'e}rique
284(2003), 181--244.

\bibitem{124} M.~Hitrik, J.~Sj\"ostrand, {\it Non-selfadjoint
perturbations of selfadjoint operators in 2 dimensions I,}
Ann. Henri
Poincar{\'e} 5(1)(2004), 1--73.

\bibitem{127} M.~Hitrik, {\it Non-selfadjoint perturbations of
selfadjoint operators in 2 dimensions II. Vanishing averages,}
Comm. Partial
Differential Equations 30(7-9)(2005), 1065--1106.

\bibitem{132} M.~Hitrik, J.~Sj\"ostrand, S.~V\~u Ng\d{o}c, {\it Diophantine tori and spectral asymptotics for non-selfadjoint
operators,} Amer. J. Math.  129(1)(2007), 105--182.

\bibitem{134}
M.~Hitrik, J.~Sj\"ostrand, {\it Non-selfadjoint perturbations of selfadjoint
operators in 2 dimensions IIIa. One branching point,}
Canad. J. Math. Vol. 60(3)(2008), 572--657.

\bibitem{139} M.~Hitrik, J.~Sj\"ostrand, {\it Rational invariant tori, phase
space tunneling, and spectra for non-selfadjoint operators in dimension 2,}
Annales Sci ENS, s\'er. 4, 41(4)(2008), 511-571.

\bibitem{141} J.~Sj\"ostrand, {\it Pseudodifferential operators and
weighted normed symbol spaces,} Serdica Mathematical Journal,
34(1)(2008), 1--38.

\bibitem{154} M.~Hitrik, J.~Sj\"ostrand, {\it Diophantine tori and Weyl laws for
  non-selfadjoint operators in dimension two,}
Comm Math Phys, Commun. Math. Phys. 314(2)(2012), 373--417.

\bibitem{164} M.~Hitrik, E.~Caliceti, S.~Graffi, J.~Sj\"ostrand {\it Quadratic
  PT--symmetric operators with real spectrum and similarity to
  self-adjoint operators,} Special issue of Journal of Physics A:
  Mathematical and Theoretical, dedicated to quantum physics with
  non-Hermitian operators, J. Phys. A: Math. Theor. 45 (2012) 444007

\bibitem{167} J.~Sj\"ostrand, G.~Uhlmann, {\it Local analytic regularity in the linearized Calder\'on problem,} preprint 2013, see http://arxiv.org/abs/1312.4065


\bibitem{Rou01} M.~Rouleux, {\it Absence of resonances for semiclassical Schr\"odinger operators with Gevrey coefficients,} Hokkaido Math. J. 30 (2001), no. 3, 475--517

\bibitem{Rou98}M.~Rouleux, {\it Resonances for a semi-classical Schr\"odinger operator near a non-trapping energy level,} Publ. Res. Inst. Math. Sci. 34 (1998), no. 6, 487--523.

\bibitem{Rou99}M.~Rouleux, {\it Tunneling effects for $h$-pseudodifferential operators, Feshbach resonances, and the Born-Oppenheimer approximation,} Evolution equations, Feshbach resonances, singular Hodge theory, 131--242, Math. Top., 16, Wiley-VCH, Berlin, 1999.

\bibitem{KaRou93} N.~Kaidi, M.~Rouleux, {\it Forme normale d'un hamiltonien \`a deux niveaux pr\`es d'un point de branchement (limite semi-classique),} C. R. Acad. Sci. Paris Sér. I Math. 317 (1993), no. 4, 359--364.

\bibitem{BeMa99} A.~Lahmar-Benbernou, A. Martinez, {\it Semiclassical asymptotics of the residues of the scattering matrix for shape resonances,} Asymptot.
Anal. 20 (1999), no. 1, 13--38.

\bibitem{BeMa02} A.~Lahmar-Benbernou, A.~Martinez, {\it On Helffer-Sj\"ostrand's theory of resonances,} Int. Math. Res. Not. 2002, no. 13, 697--717.

\bibitem{MaNaSo09} A.~Martinez, S.~Nakamura, V.~Sordoni, {\it Analytic wave front set for solutions to Schr\"odinger equations,} Adv. Math. 222 (2009), no. 4, 1277--1307.

\bibitem{MaNaSo10} A.~Martinez, S.~Nakamura, V.~Sordoni, {\it Analytic wave front set for solutions to Schr\"odinger equations II—long
range perturbations,} Comm. Partial Differential Equations 35 (2010), no. 12, 2279--2309.

\bibitem{BoFuRaZe11} J.-F.~Bony, S.~Fujii\'e, T.~Ramond, M.~Zerzeri,
  {\it Spectral projection, residue of the scattering amplitude and
    Schr\"odinger group expansion for barrier-top resonances,}
  Ann. Inst. Fourier (Grenoble) 61 (2011), no. 4, 1351--1406 (2012).



\end{thebibliography}
\end{document}